\documentclass[a4paper,11pt,reqno,twoside]{amsart}
\usepackage{amsmath}
\usepackage{amsfonts}
\usepackage{amssymb}
\usepackage{amsthm}
\usepackage{color}
\usepackage{ifpdf}
\usepackage{array}
\usepackage{mathtools}
\usepackage{url}
\usepackage{multirow,bigdelim}
\usepackage{mathdots}
\usepackage{ascmac}
\usepackage{leftidx}
\usepackage{mathrsfs}

\newcommand{\BigFig}[1]{\parbox{12pt}{\Huge #1}}
\newcommand{\BigZero}{\BigFig{0}}

\numberwithin{equation}{section}

\newtheorem{theorem}{Theorem}[section]
\newtheorem{proposition}{Proposition}[section] 
\newtheorem{lemma}[proposition]{Lemma}
\newtheorem{corollary}[proposition]{Corollary}
\newtheorem{remark}[proposition]{Remark}

\newcommand*{\C}{\mathbb{C}}
\newcommand*{\R}{\mathbb{R}}
\newcommand*{\Z}{\mathbb{Z}}

\newcommand{\comment}[1]{}

\title[Interpretation of the Schur--Cohn test]%
      {Interpretation of the Schur--Cohn test \\ in terms of canonical systems} 
\author[M. Suzuki]{Masatoshi Suzuki}
\AtBeginDocument{%
\mathtoolsset{showonlyrefs,mathic = true}
\begin{abstract}
We solve direct and inverse problems 
for two-dimensional (quasi) canonical systems 
related to exponential polynomials of a specific but sufficiently general type. 
The approach to the inverse problem in this paper provides an interpretation of 
the matrices and their determinants in the classical Schur-Cohn test 
for polynomials in terms of Hamiltonians of canonical systems.
\end{abstract}
\maketitle
}
\begin{document}

\section{Introduction}

This paper generalizes the results in \cite{Su18} by considering 
a finite-dimensional or discretized version of the theory of 
quasi-canonical systems in \cite{Su20a, Su20b}, 
but is presented in an almost self-contained fashion. 

The subject of this paper is direct and inverse problems of quasi-canonical systems, 
but we begin by stating the relation 
with the classical Schur--Cohn test obtained from the main results, 
because it may be of interest to readers in a wider field. 
On this account, we review the Schur--Cohn test 
originate from Schur~\cite{Schur17, Schur18} and Cohn \cite{Cohn22}. 
Let $f(x)=a_dx^d+a_{d-1}x^{d-1}+\cdots+a_1x+a_0$ 
be a complex polynomial of degree $d$. 
Using the triangular matrices 
\[
M_n(f) :=
\begin{bmatrix}
a_d & a_{d-1} & \cdots & a_{d-n+1} \\
     & a_d & \cdots & a_{d-n+2} \\
     &      & \ddots & \vdots \\ 
     &      &           & a_d
\end{bmatrix}, 
\quad 
N_n(f) := 
\begin{bmatrix}
a_0 & a_1 & \cdots & a_{n-1} \\
     & a_0 & \cdots & a_{n-2} \\
     &      & \ddots & \vdots \\ 
     &      &           & a_0
\end{bmatrix}, 
\]
we define the matrices
\begin{equation} \label{eq101}
L_n^\pm(f) 
:= 
\begin{bmatrix}
{}^{\rm t} M_n(f) & \pm {}^{\rm t} \overline{N_n(f)} \\
\pm N_n(f) & \overline{M_n(f)} 
\end{bmatrix}, 
\end{equation}
and denote their determinants as 
\begin{equation*} 
D_n(f) := \det L_n^\pm(f)  
\end{equation*}
for $1 \leq n \leq d$, 
where the bar means taking the complex conjugate of each entry.   
Also define $D_0(f)=1$ for convenience. 
We find that $\det L_n^+(f) = \det L_n^-(f)$ 
by multiplying each of the $(n+1)$th to $(2n)$th columns of $\det L_n^-(f)$ by $-1$ 
and then multiplying each of the $(n+1)$th to $(2n)$th rows by $-1$. 
Furthermore, $D_n(f)$ are real numbers,  
because we find that $\det L_n^+(f) = \det \overline{L_n^+(f)}$ 
by interchanging the $k$th column and the $(k+n)$th column of $\det\left[{}^{t}\left(\overline{L_n^+(f)}\right)\right]$ 
for $1\leq k\leq n$, and then interchanging the $k$th row and the $(k+n)$th row for for $1\leq k\leq n$. 

The Schur--Cohn test associates the sign changes of $D_n(f)$ with the distribution of the roots of $f$. 
Suppose that the determinants $D_n(f)$ are all different from zero 
and that the number of sign changes 
in the sequence $(D_0(f), D_1(f),\dots, D_d(f))$ is $q$. 
Then $f$ has no roots on the unit circle $\mathbb{T}=\{z \in \C~:~|z|= 1\}$ 
and $d-q$ roots inside $\mathbb{T}$ counting multiplicities. 
In particular, all roots of $f$ are inside $\mathbb{T}$ 
if and only if $D_n(f) >0$ for all $n$ 
(\cite[Corollaries 11.5.14 and 11.5.15]{RaSch02}). 
For the history and related results on the Schur--Cohn test, 
see Rahman--Schmeisser \cite[\S11.5 and pp.395--396]{RaSch02} 
or Marden \cite[\S43]{Marden66}, for example. 
English translations of \cite{Schur17} and \cite{Schur18} 
are found in \cite[pp. 31--60]{Go86} and \cite[pp. 61--88]{Go86}, respectively. 

To explain an interpretation of $D_n(f)$ in terms of quasi-canonical systems, 
we introduce the exponential polynomial
\begin{equation} \label{eq102}
E_f(z) = e^{irdz/2} f(e^{-irz}), 
\end{equation}
where $r=1$ if $d$ is even and $r=2$ if $d$ is odd. 
If all roots of $f$ are inside $\mathbb{T}$, 
the exponential polynomial 
$E_f$ belongs to the Hermite--Biehler class $\mathbb{HB}$, 
which is the class of all entire functions satisfying the inequality 
\begin{equation*} 
|E^\sharp(z)| < |E(z)| \quad \text{for all $z \in \C_+$} 
\end{equation*}
and having no real zeros, where $\C_+=\{z \in \C : \Im(z) > 0\}$. 
Then, de Branges' inverse theorem  
in the theory of canonical systems asserts that 
there exists a positive semi-definite quadratic real symmetric matrix-valued function $H_f$ 
defined on a subinterval $[t_0,t_1)$ of the real line  
such that a solution $(A(t,z),B(t,z))$ of the canonical system 
\begin{equation} \label{eq103}
-\frac{d}{dt}
\begin{bmatrix}
A(t,z) \\ B(t,z)
\end{bmatrix}
= z 
\begin{bmatrix}
0 & -1 \\ 1 & 0
\end{bmatrix}
H(t)
\begin{bmatrix}
A(t,z) \\ B(t,z)
\end{bmatrix} \quad (z \in \C)
\end{equation} 
for $H=H_f$ satisfying the boundary condition 
\begin{equation*}
\lim_{t \to t_1}\frac{\overline{A(t,z)}B(t,w)-\overline{B(t,z)}A(t,w)}{\pi(w-\bar{z})} = 0
\end{equation*} 
recovers the original $E_f$ as $E_f(z)=A(t_0,z)-iB(t_0,z)$ 
(\cite[Theorem 40]{deBranges68}). 

In \cite{Su18}, we studied a method to construct $H_f$ 
for a class of polynomials with real coefficient, 
since de Branges' inverse theorem guarantees the existence of $H_f$, 
but does not provide information about its concrete form. 
(Note that de Branges proved the inverse theorem 
by constructing the Hamiltonian of a canonical system 
in the case of polynomial function $E(z)$, 
but the above $E_f(z)$ is not a polynomial.) 
By generalizing a method in \cite{Su18} according to \cite{Su20b}, 
we present an explicit way to construct $H_f$ 
for many polynomials with complex coefficient.
As a result, we find that $H_f$ is a locally constant function of the form: 
\begin{equation*}
H_f(t) = \frac{1}{D_{n-1}(f)D_n(f)}\, \widetilde{H}_{f,n} 
\quad \text{for} \quad r(n-1)/2 \leq t < rn/2, \quad 1 \leq n \leq d, 
\end{equation*} 
where $\widetilde{H}_{f,n}$ are some positive definite matrices. 
In particular, $H_f$ is positive definite 
if all roots of $f$ are inside $\mathbb{T}$ by the Schur--Cohn test. 
This is consistent with the fact that 
a matrix-valued function $H$ obtained by de Branges' inverse theorem 
from a function of $\mathbb{HB}$ 
takes values in a set of semi-positive definite quadratic real symmetric matrices. 
Furthermore, the above method of constructing $H_f$ works 
even if $E_f$ does not belong to $\mathbb{HB}$ 
if at least $f$ and $f^\sharp=x^d\,\overline{f(1/x)}$ have no common roots, 
in which case the sign change of $H_f$ describes the distribution of the roots of $f$ 
by the Schur--Cohn test.  
This interpretation of $H_f$ by the classical result is 
what was expected in \cite[\S7.5]{Su18}. 
As the converse of the above, that is, 
by solving a direct problem of quasi-canonical systems, 
we obtain a polynomial $f$ having a specified number of roots inside $\mathbb{T}$ 
from an appropriately chosen locally constant matrix valued function $H$ 
taking values in ${\rm Sym}_2(\R) \cap {\rm SL}_2(\R)$. 
\medskip

By associating the Schur--Cohn test with 
the theory of quasi-canonical systems as described above, 
we find a correspondence between 
the set of all polynomials $f$ of degree $d$ 
with $n$ roots in $\mathbb{T}$ and $D_d(f)\not=0$ 
and the set of all sequence $(H_1,\dots,H_d)$ of $H_i \in {\rm Sym}_2(\R) \cap {\rm SL}_2(\R)$ 
in which the number of sign changes of the traces is $d-n$: 
\begin{equation*}
f~\overset{\text{inverse problem}}{\underset{\text{direct problem}}{\rightleftarrows}}~(H_1,\dots,H_d). 
\end{equation*} 
This is rigorously stated 
as a one-to-one correspondence 
by using the main theorems (Theorems \ref{thm_01}, \ref{thm_02}, and \ref{thm_03}) 
stated below and by arranging the settings appropriately. 
\medskip

To state the main results precisely, we explain the notion of quasi-canonical systems. 
Let $H(t)$ be a quadratic real symmetric matrix-valued function 
defined on a finite interval $I=[t_0,t_1)$. 
We refer to the first-order system of differential equations \eqref{eq103} 
on $I$ parametrized by $z \in \C$ as a {\it quasi-canonical system} (on $I$) 
as well as \cite{Su18} (but, as a difference, 
we deal with the additive variable $t$ instead of a multiplicative variable, 
and do not specify the condition at the right end of the interval $I$ 
when using the word). 
A column vector-valued function 
${}^{\rm t}[A(\cdot,z)\,\,B(\cdot,z)]:I \to \C^{2\times 1}$ is called a {\it solution} 
if it consists of absolutely continuous functions and satisfies \eqref{eq103} 
almost everywhere on $I$ for every fixed $z \in \C$. 
A quasi-canonical system \eqref{eq103} is called a {\it canonical system} 
if $H(t)$ is a real positive semi-definite symmetric matrix for almost all $t$, 
$H \not\equiv 0$ on any subset of $I$ with positive Lebesgue measure, 
and $H$ is locally integrable on $I$ with respect to the Lebesgue measure $dt$. 
The matrix-valued function $H$ is called a {\it Hamiltonian} of a canonical system.  
Abusing language, if it causes no confusion, 
we often call $H$ a Hamiltonian 
if a quasi-canonical system \eqref{eq103} is not a canonical system. 
\medskip

Let $d$ be a positive integer and set 
\begin{equation} \label{eq104}
(L,r):=
\begin{cases}
(d/2,1) & \text{if $d$ is even}, \\
(d,2) & \text{if $d$ is odd}. 
\end{cases}
\end{equation}
Then $2L=rd$. 
For a sequence $\mathcal{C}$ of complex numbers of length $d+1$ indexed as  
\begin{equation} \label{eq105}
\mathcal{C}= 
(C_{L},C_{L-r},C_{L-2r},\cdots,C_{-L}) \in \C^{d+1}
\quad \text{with} \quad C_LC_{-L} \not=0, 
\end{equation}
we consider the exponential polynomial
\begin{equation} \label{eq106}
E(z):=E_\mathcal{C}(z):=\sum_{j=0}^{d}C_{L-rj} e^{i(L-rj)z}
\end{equation}
along with associated functions 
\begin{equation} \label{eq107}
A(z) := A_\mathcal{C}(z) := \frac{1}{2}(E_\mathcal{C}(z)+E_\mathcal{C}^\sharp(z)), 
\quad 
B(z) := B_\mathcal{C}(z) := \frac{i}{2}(E_\mathcal{C}(z)-E_\mathcal{C}^\sharp(z)). 
\end{equation}
We also consider the polynomial 
\begin{equation*} 
f_{\mathcal{C}}(T):= \sum_{j=0}^{d} C_{-(L-rj)} \, T^{d-j} ~\in \C[T]
\end{equation*}
and denote related matrices and their determinants as 
\begin{equation} \label{eq108}
L_n^\pm({\mathcal{C}}):=L_n^\pm(f_{\mathcal{C}}), \quad 
D_n({\mathcal{C}}):=D_n(f_{\mathcal{C}}). 
\end{equation}

An exponential polynomial $E_\mathcal{C}$ of the form in \eqref{eq106} 
belongs to $\mathbb{HB}$ 
if and only if it has no zeros in the closed upper half-plane 
$\overline{\C_+}=\{z \in \C : \Im(z) \geq 0\}$ 
(\cite[Chapter VII, Theorem 6]{Levin80}).  
The latter is equivalent to the fact that $f_\mathcal{C}$ 
has no roots in the closed unit disk $\overline{\mathbb{D}}=\{z \in \C:|z|\leq 1\}$, 
since $E_\mathcal{C}$ and $f_\mathcal{C}$ 
are related as \eqref{eq102},  
$E_\mathcal{C}(z)=e^{iLz}f_\mathcal{C}(e^{-irz})$, 
by definition. 

If $E_\mathcal{C}$ belongs to $\mathbb{HB}$, 
there exists a Hamiltonian of a canonical system 
corresponding to $E_\mathcal{C}$ 
in the sense of de Branges' inverse theorem. 
In the following, we describe an explicit method for associating a Hamiltonian of a quasi-canonical system with exponential polynomial $E_\mathcal{C}$, which does not necessarily belong to $\mathbb{HB}$.
\medskip

For every $1 \leq n \leq d$, 
using the solutions of linear equations
\begin{equation} \label{eq109}
L_n^\pm({\mathcal{C}})
\begin{bmatrix}
z_n^\pm(1) \\
z_n^\pm(2) \\
\vdots \\
z_n^\pm(n) \\
\overline{z_n^\pm(n)} \\
\overline{z_n^\pm(n-1)} \\
\vdots \\
\overline{z_n^\pm(1)}
\end{bmatrix}
= \mp 
\begin{bmatrix}
0 \\
\vdots \\
0 \\
2\overline{C_{L}} \\
2C_{L} \\
0 \\
\vdots \\
0 
\end{bmatrix}
\end{equation}
for unknowns $z_n^\pm(1)\,\dots,z_n^\pm(n)$, 
where $2\overline{C_{L}}$ and $2C_{L}$ are $n$th and $(n+1)$th entries, respectively, 
we define a quadratic real symmetric matrix $H_n=H_n(\mathcal{C})$ by 
\begin{equation} \label{eq110}
\aligned 
\, & 
\begin{bmatrix} 
\Re(z_{n}^+(1)) & 
\Im(z_{n}^+(1)) \\
-\Im(z_{n}^-(1)) & 
\Re(z_{n}^-(1))
\end{bmatrix}
\cdots
\begin{bmatrix} 
\Re(z_{1}^+(1)) & 
\Im(z_{1}^+(1)) \\
-\Im(z_{1}^-(1)) & 
\Re(z_{1}^-(1))
\end{bmatrix} 
\begin{bmatrix}
0 & -1 \\ 1 & 0
\end{bmatrix} 
H_n \\
& \quad  =
\begin{bmatrix}
0 & 1 \\ -1 & 0
\end{bmatrix} 
\begin{bmatrix} 
\Re(z_{n}^+(1)) & 
\Im(z_{n}^+(1)) \\
-\Im(z_{n}^-(1)) & 
\Re(z_{n}^-(1))
\end{bmatrix}
\cdots 
\begin{bmatrix} 
\Re(z_{1}^+(1)) & 
\Im(z_{1}^+(1)) \\
-\Im(z_{1}^-(1)) & 
\Re(z_{1}^-(1))
\end{bmatrix}.
\endaligned 
\end{equation}
Then, we obtain the following results for the inverse problem 
of quasi-canonical system associated with 
exponential polynomials of the form \eqref{eq106}. 

\begin{theorem} \label{thm_01} 
Let $\mathcal{C}$ be a sequence of complex numbers of length $d+1$ as in \eqref{eq105} 
and let $(L,r)$ be as in \eqref{eq104}. 
Let $E=E_\mathcal{C}$ be the exponential polynomial defined by \eqref{eq106}. 
Suppose that $D_d(\mathcal{C})\not=0$. 
Then, 
\begin{enumerate}
\item matrices $H_n=H_n(\mathcal{C})$ of \eqref{eq110} are well-defined for all $1 \leq n \leq d$; 
\item the pair of functions $(A(t,z),B(t,z))$ defined in \eqref{eq210} below satisfies 
a quasi-canonical system \eqref{eq103}
associated with $H(t)$ defined by 
\begin{equation} \label{eq111}
H(t) = H_\mathcal{C}(t) := H_n \quad \text{for} \quad r(n-1)/2 \leq t < rn/2
\end{equation}
on the interval $t \in [0,L)$ together with the boundary conditions 
\begin{equation} \label{eq112}
\aligned 
\begin{bmatrix}
A(0,z) \\ B(0,z)
\end{bmatrix}
=
\begin{bmatrix}
A(z) \\ B(z)
\end{bmatrix},  \quad 
\lim_{t \to  L}
\begin{bmatrix}
A(t,z) \\ B(t,z)
\end{bmatrix}
= 
\begin{bmatrix}
A(0) \\ B(0)
\end{bmatrix},
\endaligned 
\end{equation}
where $A(z)$ and $B(z)$ are the functions in \eqref{eq107}; 
\item functions $A(t,z)$ and $B(t,z)$ have the forms 
\begin{equation} \label{eq113}
\aligned 
A(t,z)
&= \frac{1}{2} \sum_{j=0}^{d-n}\Bigl[\, 
a_n(L-rj)\,e^{i(L-rj-t)z} + \overline{a_n(L-rj)}\,e^{-i(L-rj-t)z} \,\Bigr], \\
B(t,z)
&= \frac{1}{2} \sum_{j=0}^{d-n}\Bigl[\,
b_n(L-rj)\,e^{i(L-rj-t)z}
 + \overline{b_n(L-rj)}\,e^{-i(L-rj-t)z} \,\Bigr], \\
\endaligned 
\end{equation}
if $r(n-1)/2 \leq t < rn/2$ and $1 \leq n \leq d$, where 
$a_{n}(k)$ and $b_{n}(k)$ 
are explicit complex numbers depending only on $\{H_n\}_{1 \leq n \leq d}$. 
\item there exist positive definite quadratic real symmetric matrices $\widetilde{H}_n$ 
such that 
\begin{equation} \label{eq114}
H_n = \frac{1}{D_{n-1}({\mathcal{C}})D_n({\mathcal{C}})} \, 
\widetilde{H}_n 
\end{equation}
holds for all $1 \leq n \leq d$.  
In particular, the positivity of $H_{\mathcal{C}}(t)$ is equivalent to 
that of $D_n(\mathcal{C})$ for all $1 \leq n \leq d$.
\end{enumerate}
\end{theorem}
\begin{remark} \label{rem_01} 
If $D_d(\mathcal{C})\not=0$, $E_\mathcal{C}$ has no real zeros 
(Lemma \ref{lem304}), especially $E_\mathcal{C}(0)\not=0$. 
Therefore, we can normalize as $E_\mathcal{C}(0)=1$ 
or equivalent $(A_\mathcal{C}(0),B_\mathcal{C}(0))=(1,0)$  
by multiplying it by an appropriate constant.
\end{remark}
We mention another way of constructing $(H(t),A(t,z),B(t,z))$ in Section \ref{section_7}. 
\medskip

As mentioned above, $E_\mathcal{C}$ of \eqref{eq106} belongs to $\mathbb{HB}$  
if and only if $f_\mathcal{C}$ has no zeros in $\overline{\mathbb{D}}$. 
The latter is equivalent that $D_n(\mathcal{C})$ 
are positive for all $1 \leq n \leq d$ by Schur--Cohn test. 
Therefore, if one of these three equivalent conditions is satisfied, 
$H(t)$ in Theorem \ref{thm_01} is defined and positive definite by \eqref{eq114}: 
\begin{corollary} \label{cor_01} 
For $\mathcal{C}$ of \eqref{eq105}, 
the following are equivalent to each other: 
\begin{enumerate}
\item $E_\mathcal{C}$ belongs to $\mathbb{HB}$; 
\item $f_\mathcal{C}$ has no roots in $\overline{\mathbb{D}}$; 
\item $D_n(\mathcal{C})>0$ for all $1 \leq n \leq d$; 
\item $H_\mathcal{C}(t)$ is positive definite for all $0 \leq t <L$. 
Thus the quasi-canonical system attached to $H_\mathcal{C}(t)$ 
is a canonical system.
\end{enumerate}
\end{corollary}

As a result of Theorem \ref{thm_01} and Corollary \ref{cor_01}, 
if an exponential polynomial $E$ of \eqref{eq106} belongs to $\mathbb{HB}$, 
it is recovered as $E(z)=A(0,z)-iB(0,z)$ 
by solving the canonical system attached to $H$ defined in \eqref{eq111},  
and, the condition at the right-endpoint in \eqref{eq112} 
guarantees that this $H$ is nothing 
but the one whose existence is stated in de Branges' inverse theorem. 

The descent of the order of $E(t,z)=A(t,z)-iB(t,z)$ 
given by Theorem \ref{thm_01} (3) 
starting from $E(z)=E(0,z)$ 
is reminiscent of the relation with the Schur transformation 
$f \mapsto \bar{a}_0 f-a_d f^\ast$ 
and Cohn's algorithm (\cite[\S11.5]{RaSch02}), 
but it is not known at present whether there is a concrete relation.  
\medskip

The converse of Theorem \ref{thm_01} is the direct problem for quasi-canonical systems \eqref{eq103} with the Hamiltonians of the form \eqref{eq111}. 
It is easier than the inverse problem, 
because the Hamiltonians is a locally constant function. 

\begin{theorem}  \label{thm_02} Let $d \in \Z_{>0}$ 
and let $(H_1,H_2,\dots,H_d)$ be a sequence 
of  matrices $H_n$ in ${\rm Sym}_2(\R) \cap {\rm SL}_2(\R)$. 
Define  a locally constant matrix-valued function $H(t)$ on $[0,L)$ by 
\begin{equation} \label{eq115}
H(t)=H_n \quad \text{for} \quad r(n-1)/2 \leq t < rn/2 \quad (1 \leq n \leq d), 
\end{equation}
where $(L,r)$ are numbers in \eqref{eq104}. 
Then the quasi-canonical system \eqref{eq103} 
associated with $H(t)$ on $[0,L)$ together with the boundary condition 
\begin{equation*} 
\lim_{t \to  L}
\begin{bmatrix}
A(t,z) \\ B(t,z)
\end{bmatrix}
= 
\begin{bmatrix}
A \\ B
\end{bmatrix} \not=0 \quad (A,B \in \R)
\end{equation*}
has a unique solution ${}^{\rm t}[A(t,z)~B(t,z)]$ whose components have the form \eqref{eq113}. 
Therefore, for $r(n-1)/2 \leq t < rn/2$, $E(t,z):=A(t,z)-iB(t,z)$ is the exponential polynomial
\begin{equation*} 
\aligned 
E(t,z) 
&= \frac{1}{2} \sum_{j=0}^{d-n}\Bigl[\, 
(a_n(L-rj)-i b_n(L-rj))e^{i(L-rj-t)z} \\
& \qquad \qquad + \overline{(a_n(L-rj)+ib_n(L-rj))} e^{-i(L-rj-t)z} \,\Bigr].
\endaligned 
\end{equation*}
Moreover, 
$E(t,0)=A-iB$ and $E(t,z)$ has no real zeros for any fixed $0 \leq t \leq L$. 
In particular, each Hamiltonian of the form \eqref{eq115}, 
namely $H_n>0$ for all $n$, 
yields an exponential polynomial 
$E(0,z)=A(0,z)-iB(0,z)$ belonging to $\mathbb{HB}$.  
\end{theorem}
\begin{remark} According to the normalization in Remark \ref{rem_01}, 
we can normalize the initial condition 
as 
$\begin{bmatrix} A \\ B \end{bmatrix} 
= \begin{bmatrix} 1 \\ 0 \end{bmatrix} $ by transformations 
$\begin{bmatrix} A(t,z) \\ B(t,z) \end{bmatrix} \mapsto 
M \begin{bmatrix} A(t,z) \\ B(t,z) \end{bmatrix}$ 
and 
$
H(t) \mapsto M H(t) M^{-1} 
$ for some $M \in {\rm GL}_2(\R)$. 
\end{remark}

The choice of intervals in \eqref{eq115} depending on the parity of $d$ 
is only adopted so that $H(t)$ has the same shape as the Hamiltonians obtained 
by solving the inverse problem as in Theorem \ref{thm_01}, 
and is not essential for solving the direct problem. 

Theorem \ref{thm_02} does not guarantee that 
the exponential polynomial $E(0,z)$ has the form \eqref{eq106}. 
In fact, $H(t)$ on $[0,1)$ with $H(t)=I_2$ for $0 \leq t < 1/2$ 
and $H(t)=-I_2$ for $1/2 \leq t < 1$ 
yields the constant function $E(0,z)=A-iB$, 
and $H(t)$ on $[0,1)$ with $H(t)=I_2$ for $0 \leq t < 1$ 
yields the function $E(0,z)=(A-iB)e^{-iz}$. 
It can be discriminated as follows 
whether $E(0,z)$ has the form \eqref{eq106}.  

\begin{theorem} \label{thm_03}
With the notation of Theorem \ref{thm_02}, 
$E(0,z)$ is an exponential polynomial of the form \eqref{eq106} 
with \eqref{eq105} if and only if 
\begin{equation*} 
(I-iJH_1)(I-iJH_2) \cdots (I-iJH_d)
\begin{bmatrix}
A \\ B
\end{bmatrix}
\end{equation*} 
is not proportional or equal to any of  three vectors
\begin{equation} \label{eq116}
{}^{\rm t}\!\begin{bmatrix} 1 & i \end{bmatrix}, 
\quad 
{}^{\rm t}\!\begin{bmatrix} 1 &  -i \end{bmatrix}, 
\quad 
{}^{\rm t}\!\begin{bmatrix} 0 & 0 \end{bmatrix}. 
\end{equation}
If $E(0,z)$ has the form \eqref{eq106} with \eqref{eq105}, 
we define $f(x)$ by $f(e^{-irz})=e^{-irdz/2}E(0,z)$. 
Then $f$ is a polynomial of degree $d$ 
and has $d-q$ roots inside $\mathbb{T}$ counting multiplicity, 
where $q$ is the number of sign changes in $(H_1,\dots,H_d)$.  
\end{theorem}

Theorem \ref{thm_03} generalizes a sufficient condition \cite[Theorem 1.5]{Su18} 
dealing with the case $A=1$, $B=0$, $H_i={\rm diag}(1/\gamma_i,\gamma_i)$. 
In fact, in that case, we have 
\begin{equation*}
\aligned 
\, & 
\begin{bmatrix} 
1 & i\gamma_{1} \\ 
-i/\gamma_{1} & 1
\end{bmatrix} 
\begin{bmatrix} 
1 & i\gamma_{2} \\ 
-i/\gamma_{2} & 1
\end{bmatrix} 
\cdots 
\begin{bmatrix} 
1 & i\gamma_{d} \\ 
-i/\gamma_{d} & 1
\end{bmatrix} 
\begin{bmatrix} 
1 \\ 0
\end{bmatrix} \\
& \quad = (\gamma_1 \gamma_2\cdots\gamma_d)^{-1}
\begin{bmatrix} 
\gamma_1(\gamma_1+\gamma_2)(\gamma_2+\gamma_3)\cdots(\gamma_{d-1}+\gamma_d) \\ 
-i(\gamma_1+\gamma_2)(\gamma_2+\gamma_3)\cdots(\gamma_{d-1}+\gamma_d)
\end{bmatrix}. 
\endaligned 
\end{equation*}
This can not be proportional to any vectors in \eqref{eq116} if $\gamma_i>0$ and $\gamma_1 \not=1$. 
\medskip

Considering Theorems \ref{thm_01}, \ref{thm_02}, and \ref{thm_03} 
together with the Schur--Cohn test, we obtain the following.

\begin{corollary} We have the one-to-one correspondence: 
\[
\left\{
\mathcal{C}=(C_{L},C_{L-r},C_{L-2r},\cdots,C_{-L}) \in \C^{d+1} ~\left\vert
\aligned 
~\cdot & ~C_LC_{-L}\not=0, \\
~\cdot & ~D_d(\mathcal{C}) \not=0, \\
~\cdot & ~\text{$f_\mathcal{C}(T)$ has $n$ roots inside $\mathbb{T}$}, \\
~\cdot & ~E_\mathcal{C}(0)=1
\endaligned 
\right.
\right\}
%
\]
\[
\text{inverse problem}~
\left\downarrow\vphantom{\int_A^B}\right. \left.\vphantom{\int_A^B}\right\uparrow
~\text{direct problem}
\]
\[
\left\{
(H_1,\dots,H_d)~\left\vert
\aligned 
~\cdot & ~H_1,\dots,H_d \in {\rm SL}_2(\R) \cap {\rm Sym}_2(\R) \\[5pt]
~\cdot & ~\text{the number of sign changes in $({\rm Tr}\, H_1,\dots, {\rm Tr}\,H_d)$ is $d-n$} \\
~\cdot & ~
\scalebox{0.9}{$
(I-iJH_1)(I-iJH_2) \cdots 
(I-iJH_d) 
\begin{bmatrix}
1 \\ 0
\end{bmatrix} 
\not=0,
~\not\in \C \begin{bmatrix} 1 \\ \pm i \end{bmatrix}
$}
\endaligned 
\right.
\right\}.
\]
\end{corollary}

\comment{
From the proof of the above results, we find that 
the correspondence 
\begin{equation*}
E_\mathcal{C} ~\mapsto~ (H_1,H_2,\dots,H_d), \quad H_n \in {\rm Sym}_2(\R) \cap {\rm SL}_2(\R), 
\end{equation*}
obtained by Theorems \ref{thm_01}, \ref{thm_02}, and \ref{thm_03} 
is one-to-one under the normalization $E(0)=1$ and the choice of intervals 
as in \eqref{eq111}. 
}

The above correspondence is compatible with the uniqueness of Hamiltonians obtained 
in de Branges' inverse theorem for entire functions in the Hermite--Biehler class. 
Hence, the exponential polynomials \eqref{eq106} belonging to the class $\mathbb{HB}$ 
are characterized in terms of the positive-deﬁniteness of Hamiltonians 
as well as the case of real coefficient in \cite{Su18}. 
\medskip

According to Corollary 1.2, there is nothing newer than the Schur--Cohn test 
regarding the criteria by which a given $E_\mathcal{C}$ belongs to  $\mathbb{HB}$, 
and 
the results \cite[Corollary 1.3, Theorems 1.6 and 1.7]{Su18} 
are reduced to the Schur--Cohn test. 
However, the method of associating $E_\mathcal{C}$ 
with Hamiltonians of quasi-canonical systems 
and the relation between the Hamiltonian $H_\mathcal{C}$ 
and determinants $D_n(\mathcal{C})$ are new. 
The former is undoubtedly important 
for direct and inverse problems for quasi-canonical systems, 
which is the subject of this paper.   
The latter shows 
the existence of an interesting class of quasi-canonical systems 
that are not necessarily canonical,  
and also contributes 
to the simplification of the proofs of the main results. 
Conversely, by proving the main theorems without using the Schur--Cohn test, 
another proof of the Schur--Cohn test may be obtained, 
but this will not be discussed in this paper. 
\medskip

To prove Theorem \ref{thm_01}, we assumed that $C_LC_{-L}\not=0$, 
but considering the relation with the Schur--Cohn test, 
it is expected that it can be removed. 
In fact, as in the case
\begin{equation*}
H_1(\mathcal{C})
 = 
\frac{1}{D_1({\mathcal{C}})}
\begin{bmatrix}
|C_{-L}-C_{L}|^2 & 2\,\Im(C_{L}C_{-L}) \\
2\,\Im(C_{L}C_{-L}) & |C_{-L}+C_{L}|^2
\end{bmatrix}, 
\end{equation*}
it is observed for small $n$ that $H_n(\mathcal{C})$ 
makes sense even if one of $C_{L}$ and $C_{-L}$ is zero. 
However, we have no idea to prove it for general $n$ 
at present.  
\medskip

The paper is organized as follows. 
We outline the proof of Theorem \ref{thm_01} in Section \ref{section_2} 
after preparing the settings similar to \cite[\S2]{Su18}, 
and complete the proof by filling in the details of Section \ref{section_2} in Section \ref{section_4}. 
The discussion in Section \ref{section_4} is a generalization of \cite{Su18}, 
but the linear equations mainly studied are changed 
(by considering a theory analogous to \cite{Su20b}), 
special matrices handled in the proof are also changed, 
and the argument of proof is largely simplified. 
In Section \ref{section_5}, we prove Theorems \ref{thm_02} and \ref{thm_03}.  
In Section \ref{section_7}, we mention an inductive way of constructing 
a triple $(H(t),A(t,z),B(t,z))$ in \eqref{eq103} 
which is different from the way of Sections \ref{section_2} and \ref{section_4}. 
The discussions of these two sections 
are straightforward generalizations of \cite[\S5-6]{Su18} according to Section \ref{section_4}. 
\medskip

\noindent
{\bf Acknowledgments}~
This work was supported by JSPS KAKENHI Grant Number JP17K05163, JP23K03050, 
and the Research Institute for Mathematical Sciences, an International Joint Usage/Research Center located in Kyoto University.

\section{Outline of the proof of Theorem \ref{thm_01}} \label{section_2}

\subsection{Hilbert spaces and operators.} 

Let $L^2(\R/(2\pi\Z))$ be the completion of 
the space of $2\pi$-periodic continuous functions on $\R$ 
with respect to the $L^2$-norm 
$\Vert f \Vert_{L^2}^2 := \langle f, f \rangle_{L^2}$, 
where 
$\langle f, g \rangle_{L^2} := (2\pi)^{-1}\int_{0}^{2\pi} f(z)\overline{g(z)}\, dz$. 
Every $f \in L^2(\R/(2\pi\Z))$ has the Fourier expansion 
$f(z) = \sum_{k \in \Z}u(k)e^{ikz}$ with $\{u(k)\}_{k \in \Z} \in l^2(\Z)$ 
and $\Vert f \Vert_{L^2}^2 = \sum_{k \in \Z}|u(k)|^2$,  
where $l^2(\Z)$ is the Hilbert space of sequences $\{u(k) \in \C \,:\, k \in \Z\}$ 
satisfying $\sum_{k \in \Z}|u(k)|^2 < \infty$. 

Fix a positive integer $d$ and set $(L,r)$ as \eqref{eq104}. 
For $t \in \R \setminus ((r/2)\Z)$, we define the vector space 
\begin{equation*}
V_t : = \left\{\left. \phi(z)= e^{-itz}f(z) + e^{itz}g(z)
\,\right|~f,\,g \in L_d^2(\R/(2\pi\Z)) \right\}
\end{equation*}
of functions of $z \in \R$, 
where 
$L_d^2(\R/(2\pi\Z))=L^2(\R/(2\pi\Z))$ if $d$ is even 
and  $L_d^2(\R/(2\pi\Z))$ is the subspace of  $L^2(\R/(2\pi\Z))$ 
consisting of  all Fourier series with odd indices $k$ if $d$ is odd. 
We define the inner product on $V_t$ by 
\begin{equation*}
\langle \phi_1, \phi_2 \rangle 
= \langle f_1, g_1 \rangle_{L^2} + \langle g_1, g_2 \rangle_{L^2}
\end{equation*}
for $\phi_j(z) = e^{-itz}f_j(z) + e^{itz}g_j(z)$ ($j=1,2$). 
Then $V_t$ with this inner product is a Hilbert space and is isomorphic to 
the (orthogonal) direct sum $L^2(\R/(2\pi\Z)) \oplus L^2(\R/(2\pi\Z))$ 
of Hilbert spaces as well as \cite[\S2]{Su18}. 
The maps $p_1:(e^{-itz}f(z) + e^{itz}g(z)) \mapsto e^{-itz}f(z)$ 
and $p_2: (e^{-itz}f(z) + e^{itz}g(z)) \mapsto e^{itz}g(z)$ 
are projections from $V_t$ to the first and the second components 
of the direct sum, respectively.
We put
\begin{equation} \label{eq201}
X(k):=e^{i(r(k+1)-1-t)z}, \quad Y(l):=e^{-i(r(l+1)-1-t)z}
\end{equation}
for $k,l \in \Z$ and $t \in \R$. 
We regard $X(k)$ and $Y(l)$ as functions of $z$, functions of $(t,z)$, 
or symbols, depending on the situation. 
For a fixed $t \in \R \setminus ((r/2)\Z)$, 
the countable set consisting of all $X(k)$ and $Y(l)$ is linearly independent over $\C$ as a set of functions of $z$, 
since the linear dependence of $\{X(k),\,Y(l)\}_{k,l \in \Z}$ implies the existence of a nontrivial pair of functions 
$f,g \in L^2(\R/(2\pi\Z))$ satisfying  $e^{-itz}f(z) + e^{itz}g(z) =0$. 
Using these vectors, $V_t$ is written as 
\begin{equation*}
\aligned 
V_t 
& = \left\{ \phi=\sum_{k \in \Z}u(k)X(k) + \sum_{l \in \Z}v(l)Y(l-r+1)
\,:~ \{u(k)\}_{k \in \Z}, \, \{v(l)\}_{l \in \Z} \in l^2(\Z) \right\}, \\
\endaligned 
\end{equation*}
and 
\begin{equation*}
\langle \phi_1, \phi_2 \rangle 
= 
\sum_{k \in \Z}u_1(k)\overline{u_2(k)} + \sum_{l \in \Z} v_1(l) \overline{v_2(l)} \quad(\phi_1,\phi_2 \in V_t), 
\end{equation*}
\begin{equation} \label{eq202}
\Vert \phi_1 \Vert^2 
= \langle \phi_1, \phi_1 \rangle 
= \sum_{k \in \Z}|u_1(k)|^2 + \sum_{l \in \Z}|v_1(l)|^2 
\quad (\phi_1 \in V_t), 
\end{equation}
if 
\begin{equation*}
\phi_i =\sum_{k \in \Z}u_i(k)X(k) + \sum_{l \in \Z}v_i(l)Y(l-r+1) \quad (i=1, 2). 
\end{equation*}
On the other hand, we have 
\begin{equation*} 
\Vert \phi \Vert^2 
= \frac{1}{2\pi}\int_{0}^{2\pi} p_1\phi(z) \overline{p_1 \phi(z)} \, dz 
+  \frac{1}{2\pi}\int_{0}^{2\pi} p_2\phi(z) \overline{p_2 \phi(z)} \, dz
\end{equation*}
for $\phi \in V_t$, since 
\begin{equation*}
\frac{1}{2\pi}\int_{0}^{2\pi}
(e^{\pm itz}f_1(z)) 
\overline{(e^{\pm itz}f_2(z))} \, dz
= \sum_{k \in \Z}u_1(k)\overline{u_2(k)} = \langle f_1, f_2  \rangle_{L^2}
\end{equation*}
for $f_j(z)=\sum_{k \in \Z}u_j(k)e^{ikz} \in L^2(\R/(2\pi\Z))$ ($j=1,2$). 
Note that, for $\phi \in V_t$, $p_1\phi$ and $p_2 \phi$ are not periodic functions of $z$, 
but the integrals $\int_I p_j\phi(z) \overline{p_j \phi'(z)} \, dz$ ($j=1,2$, $\phi, \phi' \in V_t$)
are independent of the intervals $I=[\alpha,\alpha+2\pi]$ ($\alpha \in \R$). 
We write $\phi \in V_t$ as $\phi(z)$ (respectively, $\phi(t,z)$) 
to emphasize that $\phi$ is a function of $z$ (respectively, $(t,z)$). 
If we regard $X(k)$ and $Y(l)$ as symbols, $V_t$, 
endowed with the norm defined by \eqref{eq202},
is an abstract Hilbert space isomorphic to $l^2(\Z) \oplus l^2(\Z)$.  
\medskip

For each nonnegative integer $n$, we define the closed subspace $V_{t,n}$ of $V_{t}$ 
by 
\begin{equation*}
\scalebox{0.95}{$\displaystyle{
V_{t,n} = \left\{ \phi_n=\sum_{k=0}^{\infty} u_n(k)X(k) + \sum_{l=-\infty}^{n-1} v_n(l)Y(l-r+1)
\,:~ \{u_n(k)\}_{k=0}^{\infty}, \, \{v_n(l)\}_{l=-\infty}^{n-1} \in l^2(\Z) \right\}. 
}$}
\end{equation*}
Define the projection ${\mathsf P}_n: V_t \to V_{t,n}$ by ${\mathsf P}_0=0$ and  by 
\begin{equation*}
{\mathsf P}_n \phi 
= \sum_{k=0}^{n-1} u(k)X(k) + \sum_{l=0}^{n-1} v(l)Y(l-r+1)\quad (\phi \in V_t), 
\end{equation*}
for $n \in \Z_{>0}$ (this ${\mathsf P}_n$ corresponds to ${\mathsf P}_n{\mathsf P}_n^\ast$ 
of \cite[\S2]{Su18}). 
Also define the involution ${\mathsf J}:\phi(z) \mapsto \overline{\phi(\bar{z})}$. 
Then, 
\begin{equation} \label{eq203}
\aligned 
{\mathsf J}{\mathsf P}_n \phi
& = \sum_{l=0}^{n-1} \overline{v(l)}X(l-r+1) 
+ \sum_{k=0}^{n-1}\overline{u(k)}Y(k)
\endaligned 
\end{equation}
for $\phi \in V_{t}$ and $n \in \Z_{>0}$. 
\medskip

Let $\mathcal{C} \in \C^{d+1}$ as in \eqref{eq105}.  
Using the modified function  
$
E_0(z)= e^{-i(r-1)z}E_\mathcal{C}(z)
$,  
we define two multiplication operators
\begin{equation}\label{eq204}
{\mathsf E}: \phi(z) \mapsto E_0(z) \phi(z), \quad 
{\mathsf E}^\sharp:  \phi(z) \mapsto E_0^\sharp(z) \phi(z)
\end{equation}
on $V_t$. 
These operators map $V_t$ into $V_t$, 
because ${\mathsf E}$ and ${\mathsf E}^\sharp$ are expressed as
\begin{equation*} 
{\mathsf E} = 
\sum_{j=0}^{d} C_{L-rj} {\mathsf T}_{(L-rj-r+1)/r}, 
\qquad
{\mathsf E}^{\sharp} 
=
\sum_{j=0}^{d} \overline{C_{L-rj}} \, {\mathsf T}_{-(L-rj-r+1)/r}
\end{equation*}
by using shift operators ${\mathsf T}_m:V_t \to V_t$ ($m \in \Z$) defined by 
\begin{equation*}
{\mathsf T}_m v=\sum_{k=-\infty}^{\infty} u(k)X(k+m) + \sum_{l=-\infty}^{\infty} v(l)Y(l-r+1-m). 
\end{equation*}
Both ${\mathsf E}$ and ${\mathsf E}^\sharp$ are bounded on $V_t$, 
since 
$\Vert {\mathsf E} \Vert_{\rm op} 
\leq \sum_{j=0}^{d}|C_{L-rj}|\cdot \Vert \mathsf{T}_{(L-rj-r+1)/r} \Vert_{\rm op}  \leq d M$ 
and 
$\Vert {\mathsf E}^\sharp \Vert_{\rm op} 
\leq \sum_{j=0}^{d}|C_{L-rj}|\cdot \Vert \mathsf{T}_{-(L-rj-r+1)/r} \Vert_{\rm op}  \leq d M$ 
for $M=\max\{|C_{L-rj}|\,|\, 0 \leq j \leq d\}$. 
If $E_\mathcal{C}$ has no zeros on the real line, 
${\mathsf E}$ is invertible on $V_t$ (Lemma \ref{lem301}). 
Thus the operator 
\begin{equation} \label{eq205}
\Theta := {\mathsf E}^{-1}{\mathsf E}^{\sharp}
\end{equation}
is well-defined on $V_t$, and we have 
$(\Theta \phi)(z)=(E_0^\sharp(z)/E_0(z))\phi(z)$ 
for $\phi \in V_t$. 

\subsection{Quasi-canonical systems associated with exponential polynomials.} 
Under the above settings, a quasi-canonical system associated with an exponential polynomial $E(z)$ of \eqref{eq106} 
is constructed starting from solutions of linear equations
\begin{equation} \label{eq206}
\left\{
\aligned
~( {\mathsf I} + \Theta {\mathsf J}{\mathsf P}_n )\, \phi_{n}^+
&= X(0) - \Theta Y(0), \\
~( {\mathsf I} - \Theta {\mathsf J}{\mathsf P}_n )\, \phi_{n}^-
&= X(0) + \Theta Y(0), 
\endaligned \right. 
\quad (\phi_{n}^\pm \in V_{t,n}+\Theta\mathsf{J}{\mathsf P}_nV_{t,n},\, 0 \leq n \leq d),
\end{equation}
where ${\mathsf I}$ is the identity operator. 
Note that the constant terms on the right-hand sides 
are different from that of \cite[\S2--\S3]{Su18}. 
Suppose that $D_d(\mathcal{C})\not=0$. 
Then both $\mathsf{I} \pm \Theta\mathsf{J}{\mathsf P}_n$ 
are invertible on $V_{t,n}+\Theta\mathsf{J}{\mathsf P}_nV_{t,n}$ for every $0 \leq n \leq d$, that is, 
$(\mathsf{I} \pm \Theta\mathsf{J}{\mathsf P}_n)^{-1}$ 
exist as bounded operators on $V_{t,n}+\Theta\mathsf{J}{\mathsf P}_nV_{t,n}$ (Lemma \ref{lem302}). 
Using unique solutions of \eqref{eq206}, we define
\begin{equation} \label{eq207}
\aligned 
A_n^\ast(t,z):
& = \frac{1}{2}(({\mathsf I} + {\mathsf J}){\mathsf E} \, (\phi_{n}^+ + X(0)))(t,z),  \\
B_n^\ast(t,z):
& = \frac{i}{2}(({\mathsf I} - {\mathsf J}){\mathsf E} \, (\phi_{n}^- + X(0)))(t,z).
\endaligned 
\end{equation}
The functions $A_n^\ast(t,z)$ and $B_n^\ast(t,z)$ are entire functions of $z$ 
and extend to functions of $t$ on $\R$ (by formula \eqref{eq305}). 
In particular, for $n=0$, 
\begin{equation*}
\aligned 
A_0^\ast(t,z) 
& = \frac{1}{2}\left(\mathsf{E}X(0) + \mathsf{E}^{\sharp}Y(0)\right)(t,z)
= \frac{1}{2}\left(E_0(z)e^{i(r-t-1)} + E_0^{\sharp}(z)e^{-i(r-t-1)}\right), \\ 
B_0^\ast(t,z) 
& = \frac{i}{2}\left(\mathsf{E}X(0) - \mathsf{E}^{\sharp}Y(0)\right)(t,z)
= \frac{i}{2}\left(E_0(z)e^{i(r-t-1)} - E_0^{\sharp}(z)e^{-i(r-t-1)}\right),
\endaligned 
\end{equation*}
since $\mathsf{P}_0=0$ by definition, 
and thus $A_0^\ast(0,z)=A(z)$ and $B_0^\ast(0,z)=B(z)$. 

In general, the equality 
$A_{n}^\ast(rn/2,z) =  A_{n+1}^\ast(rn/2,z)$
may not hold and the same is true about $B_n^\ast(t,z)$. 
However, we will see that the connection formula
\begin{equation} \label{eq208} 
\begin{bmatrix}
A_{n+1}^\ast(rn/2,z) \\ 
B_{n+1}^\ast(rn/2,z) 
\end{bmatrix}
= 
P_{n+1}^\ast
\begin{bmatrix}
A_{n}^\ast(rn/2,z) \\
B_{n}^\ast(rn/2,z)
\end{bmatrix}
\end{equation}
holds for some real matrix $P_{n+1}^\ast$, 
which is independent of $z$ for every $1 \leq n \leq d$ 
(Proposition \ref{prop310}). 
Therefore, 
we obtain functions $A(t,z)$ and $B(t,z)$ of $(t,z) \in [0,L) \times \C$ 
which are continuous for $t$ and entire for $z$ by defining   
\begin{equation} \label{eq209}
\begin{bmatrix}
A_n(t,z) \\
B_n(t,z)
\end{bmatrix}
:= 
P_n 
\begin{bmatrix}
A_n^\ast(t,z) \\ 
B_n^\ast(t,z) 
\end{bmatrix}
\end{equation}
for $1 \leq n \leq d$, where 
\begin{equation*} 
P_n
:= 
(P_1^\ast)^{-1}
\cdots
(P_n^\ast)^{-1}, 
\end{equation*}
and 
\begin{equation} \label{eq210}
A(t,z) := A_n(t,z), \quad B(t,z) := B_n(t,z) 
\end{equation}
for $r(n-1)/2 \leq t < rn/2$. 
We find that $A(t,z)$ and $B(t,z)$ have the form \eqref{eq113} 
(\eqref{eq209} and Lemma \ref{lem306}). 
Moreover, $(A(t,z),B(t,z))$ satisfies a quasi-canonical system \eqref{eq103} 
for the locally constant quadratic real symmetric matrix-valued function $H(t)$ defined by 
\begin{equation*} 
H(t):=H_\mathcal{C}(t):=H_n
=\begin{bmatrix}
\alpha_n & \beta_n \\ \beta_n & \gamma_n
\end{bmatrix} 
\quad \text{if} \quad r(n-1)/2 \leq a < rn/2,
\end{equation*}
where $H_n$ is defined by 
\begin{equation} \label{eq211}
\aligned 
H_n 
& =
\begin{bmatrix}
0 & 1 \\ -1 & 0
\end{bmatrix}
P_n
\begin{bmatrix}
0 & 1 \\ -1 & 0
\end{bmatrix} 
P_n^{-1}, 
\endaligned 
\end{equation}
together with the boundary conditions \eqref{eq112} 
(Proposition \ref{prop309} with \eqref{eq209} and \eqref{eq210}). 
These $H_n$ are equal to the matrices defined in \eqref{eq110} 
(Propositions \ref{prop310} and \ref{prop312}). 
Equality \eqref{eq114} is obtained by studying the solutions of equations in \eqref{eq109} 
(Proposition \ref{prop312}). 
As a summary of the above argument, we obtain Theorem \ref{thm_01}. 
See Section \ref{section_4} for details.  

On the other hand, Theorems \ref{thm_02} and \ref{thm_03} 
follow from the standard properties of quasi-canonical systems as described in Section \ref{section_5} .

%
%
\section{Proof of Theorem \ref{thm_01}.} \label{section_4}
%
%

We complete the proof of Theorem \ref{thm_01} in this section 
by filling in the details of the outline 
described in the previous section. 
We fix $d \in \Z_{>0}$ and a sequence $\mathcal{C} \in \C^{d+1}$ as in \eqref{eq105} throughout this section. 

\begin{lemma} \label{lem301}
Let ${\mathsf E}$ be the multiplication operator defined by \eqref{eq204} 
for $E=E_\mathcal{C}$. 
Suppose that $E$ has no real zeros. 
Then ${\mathsf E}$ is invertible on $V_t$, 
and thus $\Theta$ of \eqref{eq205} 
is well-defined as a bounded operator on $V_t$.  
Moreover $\Vert \Theta \Vert_{\rm op}=1$. 
\end{lemma}
\begin{proof}  
It is sufficient to prove that ${\mathsf E}$ is invertible on $L_d^2(\R/(2\pi\Z))$, 
since $V_t$ is a direct sum of $e^{\pm itz}L_d^2(\R/(2\pi\Z))$ 
and $e^{\pm 2 itz}(1/E_0(z))f(z)=g(z)$ is impossible for any $0 \not=f,g \in L_d^2(\R/(2\pi\Z))$. 
We have $1/E_0(z) \in L^\infty(\R/(2\pi\Z))$ by assumption. 
Therefore, multiplication by $1/E_0(z)$ 
defines a bounded operator ${\mathsf E}^{-1}$ on $L^2(\R/(2\pi\Z))$ 
with the norm $\Vert {\mathsf E}^{-1} \Vert_{\rm op}=\Vert 1/E_0 \Vert_{L^\infty}$. 
Moreover $\Vert \Theta \Vert_{\rm op}=\Vert E_0^\sharp/E_0 \Vert_{L^\infty}=1$. 
Hence the case of even $d$ is proved. 
For odd $d$, we find that $(1/E_0(z))f(z) \in L_d^\infty(\R/(2\pi\Z))$ 
for $f \in L_d^\infty(\R/(2\pi\Z))$, 
since $f(z)=E_0(z)g(z)$ is impossible 
for a Fourier series $g \in L^2(\R/(2\pi\Z))$ containing $e^{ikz}$ of an even index. 
Hence the claim holds as well. 
\end{proof} 

\begin{lemma} \label{lem302} 
Let $t \not \in (r/2)\Z$. 
Suppose that $E=E_\mathcal{C}$ has no real zeros.
Then, $\Theta\mathsf{J}\mathsf{P}_n$ 
defines a compact anti-linear (conjugate linear)  
operator on $V_{t,n}+\Theta\mathsf{J}{\mathsf P}_nV_{t,n}$ for each $0 \leq n \leq d$. 
Additionally suppose that $D_d({\mathcal{C}})\not=0$. 
Then, $\mathsf{I} \pm \Theta\mathsf{J}{\mathsf P}_n$ 
are invertible on $V_{t,n}+\Theta\mathsf{J}{\mathsf P}_n V_{t,n}$, and 
\eqref{eq206} have unique solutions in $V_{t,n}+\Theta\mathsf{J}{\mathsf P}_nV_{t,n}$ 
for each $0 \leq n \leq d$. 
\end{lemma}
\begin{remark}
If $r \in r\Z/2$, $\mathsf{I}\pm\Theta\mathsf{J}\mathsf{P}_n$ may not be invertible. 
\end{remark}
\begin{proof} 
The assertion is trivial for $n=0$, since ${\mathsf P}_0=0$ as an operator. 
Let $n \geq 1$ and write $W_n=V_{t,n}+\Theta\mathsf{J}{\mathsf P}_nV_{t,n}$.  
By definition, $\mathsf{P}_n$ is a projection from $V_t$ into $V_{t,n}$, 
so $\Theta\mathsf{J}{\mathsf P}_n$ is an operator on $W_n$.  
The image of $W_n$ by $\mathsf{E}^\sharp\mathsf{J}\mathsf{P}_n$ 
is finite dimensional by definition of $\mathsf{E}^\sharp$ and \eqref{eq203}, 
thus $\Theta\mathsf{J}\mathsf{P}_n
=\mathsf{E}^{-1}(\mathsf{E}^\sharp\mathsf{J}\mathsf{P}_n)$ 
is a finite rank operator which is compact. 
On the other hand, 
$\Vert \Theta\mathsf{J}{\mathsf P}_n \Vert_{\rm op} 
\leq \Vert \Theta\Vert_{\rm op}\cdot \Vert\mathsf{J}{\mathsf P}_n \Vert_{\rm op}
\leq  1$. 
Therefore, 
if $\Theta\mathsf{J}{\mathsf P}_n|_{W_n}$
has no eigenvalues of modulus one, 
$\Vert \Theta\mathsf{J}{\mathsf P}_n|_{W_n} \Vert_{\rm op}<1$ 
and thus 
$\mathsf{I} \pm \Theta\mathsf{J}{\mathsf P}_n$ 
are invertible on $W_n$ by the convergence of Neumann series. 

Assume that $\Theta{\mathsf J}{\mathsf P_n} \phi
 = \lambda \phi$ and $|\lambda|=1$ 
for $\phi \in W_n$.
Because $\Theta$ is an isometry on $V_t$ by Lemma \ref{lem301}, 
we have $\Vert \Theta {\mathsf J}{\mathsf P_n} \phi \Vert
 = \Vert \phi \Vert$, and 
$
\Vert \Theta{\mathsf J}{\mathsf P_n} \phi \Vert^2 
 = \Vert {\mathsf J}{\mathsf P_n} \phi \Vert^2 
 = \sum_{k=0}^{n-1} |u(k)|^2 + \sum_{l=0}^{n-1} |v(l)|^2
$
by \eqref{eq203}, while 
$ \Vert \phi \Vert^2 = \sum_{k \in \Z} |u(k)|^2 + \sum_{l \in \Z} |v(l)|^2$. 
Thus, 
\begin{equation*} 
\phi = \sum_{k=0}^{n-1} u_n(k)X(k) + \sum_{l=0}^{n-1} v_n(l)Y(l-r+1).
\end{equation*} 
For such $\phi$, 
\begin{equation*} 
\aligned 
\mathsf{E}^{\sharp}{\mathsf J}{\mathsf P_n}\phi
& = \sum_{j=0}^{d} \sum_{l=0}^{n-1} \overline{C_{-(L-rj)}} 
\overline{v(l)}X(l+(d-r+1)/2-j) \\
& \quad + \sum_{j=0}^{d} \sum_{k=0}^{n-1} \overline{C_{-(L-rj)}} 
\overline{u(k)}Y(k-(d+r-1)/2+j) 
\endaligned 
\end{equation*} 
and
\begin{equation*} 
\aligned 
\mathsf{E} \phi
& = \sum_{j=0}^{d} \sum_{k=0}^{n-1} C_{L-rj} u(k) X(k+(d-r+1)/2-j) \\
& \quad + \sum_{j=0}^{d} \sum_{l=0}^{n-1}C_{L-rj} v(l) Y(l-(d+r-1)/2+j).
\endaligned 
\end{equation*} 
Comparing $2n$ coefficient of $X(k)$ 
with indices 
$-(d+r-1)/2 \leq k \leq -(d+r-1)/2+n-1$ 
and 
$(d-r+1)/2 \leq k \leq (d-r+1)/2+n-1$ 
in the equality 
$\lambda \mathsf{E}\phi-\mathsf{E}^{\sharp}\mathsf{J}\mathsf{P}_n \phi =0$, 
we obtain the linear equation
\begin{equation} \label{eq301}
M_\lambda \cdot
{}^{t}\! \begin{bmatrix}u(0) & \cdots & u(n-1) & 
\overline{v(0)} & \cdots & \overline{v(n-1)} \end{bmatrix} =0, 
\end{equation}
where 
$\displaystyle{
M_\lambda=
\begin{bmatrix}
\lambda \cdot {}^{\rm t}\!M_n({\mathcal{C}}) & \overline{{}^{\rm t}\!N_n({\mathcal{C}})} \\
\lambda \cdot  N_n({\mathcal{C}}) & \overline{M_n({\mathcal{C}})}
\end{bmatrix}. 
}
$
Here, $\det M_\lambda \not=0$ by assumption for $D_d({\mathcal{C}})$. 
Therefore \eqref{eq301} has no nontrivial solutions, which implies $\phi=0$. 
Consequently, none of $\lambda \in \C$ with modulus $1$ is an eigenvalue 
of $\Theta\mathsf{J}{\mathsf P}_n|_{W_n}$, 
and hence complete the proof. 
\end{proof}

\begin{lemma} \label{lem304}
Let $E=E_\mathcal{C}$. 
\begin{enumerate}
\item Suppose that $D_d(\mathcal{C})\not=0$. 
Then $E$ and $E^\sharp$ have no common zeros. 
In particular, $E$ has no real zeros. 
\item Suppose that $E$ belongs to the Hermite--Biehler class $\mathbb{HB}$. 
Then $D_d(\mathcal{C})\not=0$. 
\end{enumerate}
\end{lemma}
\begin{proof} 
The determinant $D_d(\mathcal{C})$ is zero 
if and only if $f_\mathcal{C}(T)$ and $f_\mathcal{C}^\sharp(T):=T^{d}\overline{f_\mathcal{C}(T^{-1})}$ 
have a common root (\cite[Lemmas 11.5.11 and 11.5.12]{RaSch02}). 
The latter is equivalent that 
$E$ and $E^\sharp$ have a common zero, 
since $E(z)=e^{iLz}f_\mathcal{C}(e^{-irz})$ 
and $E^\sharp(z)=e^{iLz}f_\mathcal{C}^\sharp(e^{-irz})$. 
In general, if an entire function $F(z)$ has a real zero, 
it is also a zero of $F^\sharp(z)$. 
Hence (1) holds. 
If $E$ belongs to $\mathbb{HB}$, it has no real zeros and 
$|E(\bar{z})|<|E(z)|$ in $\C_+$ by definition of $\mathbb{HB}$. 
Therefore $E$ and $E^\sharp$ have no common zeros. 
Hence (2) holds. 
\end{proof}

In the remaining part of this section, 
we assume that $\mathcal{C}$ is taken as in \eqref{eq105} and satisfies 
\begin{equation*}
D_d(\mathcal{C})\not=0
\end{equation*}
so that both $\mathsf{I} \pm \Theta\mathsf{J}\mathsf{P}_n$ 
are invertible on $V_{t,n}+\Theta\mathsf{J}{\mathsf P}_nV_{t,n}$ for every $0 \leq n \leq d$ 
by Lemmas \ref{lem301} and \ref{lem302}. 
Note that this assumption is satisfied if $E_\mathcal{C}$ belongs to $\mathbb{HB}$ 
by Lemma \ref{lem304}. 

Under the above assumption, we consider the equations 
\begin{equation} \label{eq302}
\left\{ \aligned
~( {\mathsf E} + {\mathsf E}^{\sharp} {\mathsf J}{\mathsf P}_n )\, \phi_{n}^+
&= {\mathsf E}X(0) - {\mathsf E}^{\sharp} Y(0), \\
~( {\mathsf E} - {\mathsf E}^{\sharp} {\mathsf J}{\mathsf P}_n )\, \phi_{n}^-
&= {\mathsf E}X(0) + {\mathsf E}^{\sharp} Y(0), 
\endaligned
\right.
\quad (\phi_{n}^\pm \in V_{t,n}+\Theta\mathsf{J}{\mathsf P}_nV_{t,n} ,\, 0 \leq n \leq d),
\end{equation}
which is equivalent to \eqref{eq206}, since ${\mathsf E}$ is invertible.  
Firstly, we note that 
each $\phi \in V_{t,n}+\Theta\mathsf{J}{\mathsf P}_nV_{t,n}$ has the absolutely convergent expansion 
\begin{equation*} 
\phi = \sum_{k=0}^{\infty} u(k)X(k) 
+ \sum_{l=-\infty}^{n-1} v(l)Y(l-r+1)
\end{equation*}
as a function of $z$ if $\Im(z)>0$ is large enough. This is trivial for $\phi \in V_{t,n}$ 
and follows for $\phi \in \Theta\mathsf{J}{\mathsf P}_nV_{t,n}$ from \eqref{eq203} 
and the expansion 
\begin{equation*} 
\frac{E_0^\sharp (z)}{E_0(z)}
= \frac{\overline{C_{L}} + \sum_{j=1}^{d} \overline{C_{L-rj}} e^{irjz}}
{C_{-L}+\sum_{j=1}^{d} C_{-(L-rj)} e^{irjz}} \cdot e^{2i(r-1)z}  
= e^{2i(r-1)z} \sum_{m=0}^{\infty} \widetilde{C}_m e^{irmz} 
\end{equation*}
that holds if $\Im(z)>0$ is large enough. 
Secondly, we introduce several special matrices to study \eqref{eq302}. 
We define the square matrix $\mathfrak{E}_0$ of size $8d$ by 
\begin{equation*}
\mathfrak{E}_0=\mathfrak{E}_0(\mathcal{C}):=
\left[
\begin{array}{cc|cc|cc|cc}
 & & & & & & & \\
 \multicolumn{2}{c|}{\raisebox{1.5ex}[0pt]{$\mathfrak{e}_0(\mathcal{C}) $}} & & & & & &  \\ \hline
 & & & & & & &  \\
 & &  \multicolumn{2}{c|}{\raisebox{1.5ex}[0pt]{$\mathfrak{e}_0(\mathcal{C}) $}}  & & & & \\ \hline
 & & & & & & & \\
 & & & & \multicolumn{2}{c|}{\raisebox{1.0ex}[0pt]{${}^{\rm t}\overline{\mathfrak{e}_0(\mathcal{C})}$}} & & \\ \hline
 & & & & & & & \\
 & & & & & & \multicolumn{2}{c}{\raisebox{1.0ex}[0pt]{${}^{\rm t}\overline{\mathfrak{e}_0(\mathcal{C})}$}}
\end{array}
\right],
\end{equation*}
where $\mathfrak{e}_0(\mathcal{C}) $ is the lower triangular matrix of size $2d$ defined by
\begin{equation*}
\scalebox{0.9}{$
\mathfrak{e}_0=\mathfrak{e}_0(\mathcal{C}) 
:= 
\left[
\begin{array}{llllllll}
C_{-L} & & & & & & & \\
C_{-L+r} & \ddots & & & & & & \\
\vdots & \ddots & C_{-L} & & & & & \\
C_{L-r} & \ddots & C_{-L+r} & C_{-L} & & & & \\
C_{L} & \ddots & \vdots & C_{-L+r} & C_{-L} & & & \\
0 & \ddots & C_{L-r} & \vdots & \ddots & \ddots & & \\
\vdots & \ddots & C_{L} & C_{L-r} & \ddots  & C_{-L+r} & C_{-L} & \\
0 & \cdots & 0 & C_{L} & C_{L-r} & \cdots  & C_{-L+r} & C_{-L} \\
\end{array}
\right],
$}
\end{equation*}
and define the square matrix $\mathfrak{E}_n^\sharp$ of size $8d$ by 
\begin{equation*}
\mathfrak{E}_n^\sharp 
:=\mathfrak{E}_n^\sharp (\mathcal{C}):=
\left[
\begin{array}{cc|cc|cc|cc}
 & & & & & & & \\
  & & & & \multicolumn{2}{c|}{\raisebox{1.5ex}[0pt]{$
\overline{\mathfrak{e}_{2,n}(\mathcal{C})} $}} & &  \\ \hline
 & & & & & & &  \\
 & &  & & & &  \multicolumn{2}{c}{\raisebox{1.2ex}[0pt]{$
\overline{\mathfrak{e}_{1,n}(\mathcal{C})} $}} \\ \hline
 & & & & & & & \\
\multicolumn{2}{c|}{\raisebox{1.2ex}[0pt]{$
J_{2d}  \cdot \mathfrak{e}_{1,n}(\mathcal{C}) \cdot J_{2d} 
$}} & & & &  & & \\ \hline
 & & & & & & & \\
 & & \multicolumn{2}{c|}{\raisebox{1.0ex}[0pt]{$
J_{2d} \cdot \mathfrak{e}_{2,n}(\mathcal{C}) \cdot J_{2d}  
$}} & & & & 
\end{array}
\right]
\end{equation*}
with 
\begin{equation*}
\aligned 
\mathfrak{e}_{1,n}
& = \mathfrak{e}_{1,n}(\mathcal{C}) 
:= 
\left[
\begin{array}{cc|lll|ccc}
 & & C_{L} & & & & \\
 & &  C_{L-r} & \ddots & & & \\
 & & \vdots & \ddots & C_{L} & & \\
 \multicolumn{2}{c|}{\raisebox{1.5ex}[0pt]{\BigZero}}  & C_{-L+r} & \ddots 
& C_{L-r} &  \multicolumn{2}{c}{\raisebox{1.5ex}[0pt]{\BigZero}} \\
 & & C_{-L} & \ddots & \vdots & & \\
 & & & \ddots & C_{-L+r} & & \\
 & & & & C_{-L} &  &  \\ \hline
 & & & & &  & \\
 \multicolumn{2}{c|}{\raisebox{1.5ex}[0pt]{\BigZero}} & \multicolumn{3}{c|}{\raisebox{1.5ex}[0pt]{\BigZero}} 
& \multicolumn{2}{c}{\raisebox{1.5ex}[0pt]{\BigZero}} \\
\end{array}
\right]=
\begin{array}{c|c|c|c}
 d & n & d-n & \\ \hline
 & & & d+n  \\ \hline
 & & & d-n 
\end{array}, \\
\mathfrak{e}_{2,n}
&= \mathfrak{e}_{2,n}(\mathcal{C}) 
:= 
\left[
\begin{array}{cc|lll}
& & & & \\
 \multicolumn{2}{l|}{\raisebox{1.5ex}[0pt]{\BigZero}} & \multicolumn{3}{c}{\raisebox{1.5ex}[0pt]{\BigZero}}  \\\hline
 & & C_{L} &  & \\
 & & C_{L-r} & \ddots &  \\
 & & \vdots & \ddots & C_{L} \\ 
\multicolumn{1}{c}{\raisebox{1.5ex}[0pt]{\BigZero}} & & C_{-L+r}  & \ddots & C_{L-r} \\
& & C_{-L}  & \ddots & \vdots \\
& & & \ddots & C_{-L+r}  \\ 
& & & & C_{-L}
\end{array}
\right]=
\begin{array}{c|c|c}
2d-n & n & \\ \hline
  & & d-n  \\ \hline
  & & d+n 
\end{array},
\endaligned 
\end{equation*}
where the right-hand sides mean the size of each block of matrices in middle terms 
and $J_n$ is the anti-diagonal matrix of size $n$: 
\begin{equation*}
J_{n}
=\begin{bmatrix}
 & & 1 \\
 & \iddots & \\
1  & & 
\end{bmatrix}.
\end{equation*}
We also define the column vector $\chi$ of length $8d$ by 
\begin{equation*}
\chi = \chi_{8d} ={}^{\rm t}\! \begin{bmatrix} 1& 0 & \cdots & 0 & 1\,\, \end{bmatrix}
\end{equation*}
and 
\begin{equation*}
\mathfrak{J} =
\mathfrak{J}^{(8d)} = 
\begin{bmatrix}
 & I_{4d} \\
I_{4d} & 
\end{bmatrix}. 
\end{equation*}
where $I_{4d}$ is the identity matrix of size $4d$. 
\medskip

Let $\phi_n^{\pm}=\sum_{k=0}^{\infty} u_n^{\pm}(k)X(k) 
+ \sum_{l=-\infty}^{n-1} v_n^{\pm}(l)Y(l-r+1)$ 
be absolutely convergent expansions of 
the solutions of \eqref{eq302} for $0 \leq n \leq d$, 
where it is assumed that $\Im(z)>0$ is large enough.
Using these coefficient of $\phi_n^{\pm}$ and putting  
\begin{equation*} 
v_n^{\pm}(n)=v_n^{\pm}(n+1)\cdots=v_n^{\pm}(d-1)=0
\end{equation*}
if $0 \leq n \leq d-1$, 
we define the column vectors $\Phi_{n}^{\pm}$ of length $8d$ by
\begin{equation} \label{eq303}
\Phi_{n}^{\pm} 
 = 
\begin{bmatrix}
\Phi_{n,1}^{\pm} \\
\Phi_{n,2}^{\pm} \\
J_{2d}\cdot\overline{\Phi_{n,2}^{\pm}} \\
J_{2d}\cdot\overline{\Phi_{n,1}^{\pm}} 
\end{bmatrix}, \quad 
\left\{
\aligned 
\Phi_{n,1}^{\pm} 
& = {}^{\rm t}\!\begin{bmatrix} 
u_n^{\pm}(0) & u_n^{\pm}(1) & \cdots & u_n^\pm(2d-1) 
\end{bmatrix}, \\
\Phi_{n,2}^{\pm} 
& = {}^{\rm t}\!\begin{bmatrix} 
v_n^{\pm} (d-1) & v_n^{\pm}(d-2) &\cdots & v_n^{\pm}(-d) 
\end{bmatrix}.
\endaligned
\right. 
\end{equation}
Substituting the above expansion of $\phi_n^{\pm}$ into \eqref{eq302}, 
we obtain linear equations 
\begin{equation} \label{eq304}
(\mathfrak{E}_0 \pm \mathfrak{E}_n^\sharp) \cdot \Phi_{n}^{\pm} 
 = \mathfrak{E}_0 \cdot \chi \mp \mathfrak{J}\cdot \mathfrak{E}_0 \cdot  \chi \\
\quad (0 \leq n \leq d) 
\end{equation}
by comparing coefficient of $X(k)$ and $Y(l)$ 
for $-(d+r-1)/2 \leq k, l \leq (3d-r-1)/2$, 
and
\begin{equation*} 
\sum_{j=0}^{d} C_{-(L-rj)} u_n^{\pm}(J_+ -j) = 0, \quad  
\sum_{j=0}^{d} C_{-(L-rj)} v_n^{\pm}(J_-+j) = 0 
\end{equation*}
for every $J_+ \geq 2d$ and $J_- \leq -d-1$ 
by comparing other coefficient.  
\begin{lemma} \label{lem305}
Let $0 \leq n \leq d$.  
Then $\det(\mathfrak{E}_0 \pm \mathfrak{E}_n^\sharp)\not=0$ 
if $\mathsf{I} \pm \Theta\mathsf{J}\mathsf{P}_n$ 
is invertible on $V_{t,n}+\Theta\mathsf{J}\mathsf{P}_n V_{t,n}$ 
or equivalently $D_n(\mathcal{C})\not=0$. 
\end{lemma} 
\begin{proof} 
Let 
\[
\aligned 
\mathbf{k} & =(n+1,n+2,\dots,2d; 2d+n+1,2d+n+2,\dots,4d; \\
& \qquad \qquad \qquad \qquad 
4d+1, 4d+2, \dots 6d-n; 6d+1, 6d+2, \dots, 8d-n) 
\endaligned 
\]
be a list of indices of columns of 
$\mathfrak{E}_0 \pm \mathfrak{E}_n^\sharp$ and let  
\[
\aligned 
\mathbf{k}_1=(2d; 4d; 4d+1; 6d+1), ~
\mathbf{k}_2=(2d-1; 4d-1; 4d+2; 6d+2),~\dots
\endaligned 
\]
be sublists of $\mathbf{k}$. 
Expanding $\det(\mathfrak{E}_0 \pm \mathfrak{E}_n^\sharp)$ 
with respect to columns with indices $\mathbf{k}_1$, $\mathbf{k}_2$, $\dots$ in this order, 
we have
\[
\det(\mathfrak{E}_0 \pm \mathfrak{E}_n^\sharp)
= |C_{-L}|^{4(2d -n)}D_n(\mathcal{C})^2
\]
Therefore, we obtain the conclusion by Lemma \ref{lem302}, 
since $C_{-L}\not=0$ by assumption. 
\end{proof}

On the other hand, by \eqref{eq302}, we have
\[
\aligned
{\mathsf E} \, \phi_{n}^{\pm} 
&= {\mathsf E}X(0) \mp {\mathsf E}^{\sharp} Y(0) 
\mp {\mathsf E}^{\sharp} {\mathsf J}{\mathsf P}_n  \phi_{n}^{\pm} \\
&= \sum_{j=0}^{d} C_{L-rj} X((d-r+1)/2-j) \mp 
\sum_{j=0}^{d} \overline{C_{L-rj}}  Y((d-r+1)/2-j) \\
& \quad \mp 
\sum_{j=0}^{d}\sum_{l=0}^{n-1} \overline{C_{L-rj}} \overline{v_n^\pm(l)} X(l-(d+r-1)/2+j) \\
& \quad \mp 
\sum_{j=0}^{d}\sum_{k=0}^{n-1} \overline{C_{L-rj}} \overline{u_n^\pm(k)} Y(k+(d-r+1)/2-j) . 
\endaligned
\]
Therefore, we can write 
\begin{equation} \label{eq305}
{\mathsf E}(\phi_{n}^{\pm} + X(0))  
 =\sum_{k=-(d+r-1)/2}^{(d-r+1)/2+n-1}\Bigl( p_n^{\pm}(k)X(k) + q_n^{\pm}(k)Y(k) \Bigr) \quad 
\end{equation}
for some complex numbers $p_n^{\pm}(k)$ and $q_n^{\pm}(k)$. 
Hence 
$({\mathsf E}(\phi_{n}^{\pm}+X(0)))(t,z)$ extend to smooth functions of $t$ on $\R$ 
by the right-hand side of \eqref{eq305}. 
We use the same notation for such extended functions. 

We put $p_n^{\pm}(k)=q_n^{\pm}(k)=0$ for every $d+n \leq k \leq 2d-1$ if $0 \leq n \leq d-1$ and 
define the column vectors $\Psi_n^{\pm}$ of length $8d$ by
\[
\Psi_{n}^{\pm} 
= 
\begin{bmatrix}
\Psi_{n,1}^{\pm} \\
\Psi_{n,2}^{\pm} \\
J_{2d} \cdot \overline{\Psi_{n,2}^{\pm}} \\
J_{2d} \cdot \overline{\Psi_{n,1}^{\pm}}
\end{bmatrix}, \quad 
\left\{
\aligned
\Psi_{n,1}^{\pm} 
&={}^{\rm t}\!
\begin{bmatrix} 
p_n^{\pm}(-\tfrac{d+r-1}{2}) & p_n^{\pm}(-\tfrac{d+r-3}{2}) & 
\cdots & p_n^{\pm}(\tfrac{3d-r-1}{2}) 
\end{bmatrix}, \\
\Psi_{n,2}^{\pm} 
&={}^{\rm t}\!
\begin{bmatrix} 
q_n^{\pm}(\tfrac{3d-r-1}{2}) & q_n^{\pm}(\tfrac{3d-r-3}{2}) & 
\cdots & q_n^\pm(-\tfrac{d+r-1}{2})
\end{bmatrix}.
\endaligned
\right.
\]
Then we have
\begin{equation} \label{eq306}
\Psi_n^{\pm} = \mathfrak{E}_{0}(\Phi_n^{\pm} + \chi) 
\end{equation}
by comparing the right-hand sides of \eqref{eq305} with 
\[
\aligned
{\mathsf E}(\phi_{n}^{\pm} + X(0))
&= \sum_{j=0}^{d}  \sum_{k=0}^{\infty} C_{L-rj} u_n^\pm(k) X(k+(d-r+1)/2-j) \\
& \quad 
+ \sum_{j=0}^{d} \sum_{l=-\infty}^{n-1} C_{L-rj}v_n^\pm(l) Y(l-(d+r-1)/2+j) \\
& \quad 
+ \sum_{j=0}^{d}  C_{L-rj} X((d-r+1)/2-j), 
\endaligned
\]
which is obtained by a direct calculation of the left-hand side 
with the expansion $\phi_n^{\pm}=\sum_{k=0}^{\infty} u_n^{\pm}(k)X(k) 
+ \sum_{l=-\infty}^{n-1} v_n^{\pm}(l)Y(l-r+1)$. 

\begin{lemma} \label{lem306} 
Let $0 \leq n \leq d$. Then 
$p_n^{\pm}(k) \pm \overline{q_n^{\pm}(k)}=0$
if $(d-r+1)/2+1 \leq k \leq (d-r+1)/2+n-1$ or $-(d+r-1)/2 \leq k \leq -(d+r-1)/2+n-1$, 
and $p_n^{\pm}((d-r+1)/2) \pm \overline{q_n^{\pm}((d-r+1)/2)}=2C_{L}$. 
Therefore, 
\begin{equation} \label{eq307}
\aligned 
(\mathsf{I} & \pm \mathsf{J}){\mathsf E} (\phi_{n}^{\pm}+X(0)) \\
& \quad = 
\sum_{k=-(d+r-1)/2+n}^{(d-r+1)/2}\Bigl( 
(p_n^{\pm}(k)\pm \overline{q_n^{\pm}(k)})X(k) 
\pm  
(\overline{p_n^{\pm}(k)} \pm q_n^{\pm}(k))Y(k) \Bigr).
\endaligned 
\end{equation}
\end{lemma}
\begin{proof} 
Linear equations \eqref{eq304}
are equivalent to 
$(\mathfrak{E}_0 \pm \mathfrak{E}_n^{\sharp}) 
(\Phi_n^{\pm}+\chi) = 
2 \mathfrak{E}_0 \cdot \chi$. 
They are written in matrix forms as 
\begin{equation} \label{eq308}
\scalebox{0.925}{$
\aligned 
\, & 
\begin{bmatrix}
\mathfrak{e}_0 & & \pm \overline{\mathfrak{e}_{2,n}} & \\
& \mathfrak{e}_0 & & \pm \overline{\mathfrak{e}_{1,n}} \\
\pm J_{2d}\cdot\mathfrak{e}_{1,n}\cdot J_{2d} & & {}^{\rm t}\overline{\mathfrak{e}_0} & \\
& \pm J_{2d}\cdot\mathfrak{e}_{2,n}\cdot J_{2d} & & {}^{\rm t}\overline{\mathfrak{e}_0} 
\end{bmatrix}
(\Phi_n^\pm + \chi)
= 2
\begin{bmatrix}
\mathfrak{e}_0 & &  & \\
& \mathfrak{e}_0 & &  \\
& & {}^{\rm t}\overline{\mathfrak{e}_0} & \\
& & & {}^{\rm t}\overline{\mathfrak{e}_0} 
\end{bmatrix} \! \chi
\endaligned 
$}
\end{equation}
by definition of matrices. On the other hand, we have
\[
\scalebox{0.925}{$
\begin{bmatrix}
\mathfrak{e}_0 & & \pm {}^{\rm t}\overline{\mathfrak{e}_0} & \\
& \mathfrak{e}_0 & & \pm {}^{\rm t}\overline{\mathfrak{e}_0} \\
\pm \mathfrak{e}_{0} & & {}^{\rm t}\overline{\mathfrak{e}_0} & \\
& \pm \mathfrak{e}_{0} & & {}^{\rm t}\overline{\mathfrak{e}_0} 
\end{bmatrix}
 (\Phi_n^{\pm}+\chi)
 = 
\begin{bmatrix} 
R_n^\pm \\
\pm J_{2d}\cdot \overline{R_n^\pm} \\
\pm R_n^\pm \\
J_{2d}\cdot \overline{R_n^\pm} 
\end{bmatrix}, 
\quad 
R_n^\pm
= \begin{bmatrix} 
p_n^{\pm}(-\tfrac{d+r-1}{2}) \pm \overline{q_n^{\pm}(-\tfrac{d+r-1}{2})} \\ 
p_n^{\pm}(-\tfrac{d+r-3}{2}) \pm \overline{q_n^{\pm}(-\tfrac{d+r-3}{2})}  \\ \vdots \\ 
p_n^{\pm}(\tfrac{3d-r-1}{2}) \pm \overline{q_n^{\pm}(\tfrac{3d-r-1}{2})} \\ 
\end{bmatrix},
$}
\]
since 
$\displaystyle{
(I_{8d} \pm \mathfrak{J}) \cdot \mathfrak{E}_0 
=
\begin{bmatrix}
\mathfrak{e}_0 & & \pm {}^{\rm t}\overline{\mathfrak{e}_0} & \\
& \mathfrak{e}_0 & & \pm {}^{\rm t}\overline{\mathfrak{e}_0} \\
\pm \mathfrak{e}_{0} & & {}^{\rm t}\overline{\mathfrak{e}_0} & \\
& \pm \mathfrak{e}_{0} & & {}^{\rm t}\overline{\mathfrak{e}_0} 
\end{bmatrix}
}$ 
by definition of $\mathfrak{E}_{0}$ and 
\[
(I_{8d} \pm \mathfrak{J})\mathfrak{E}_0 (\Phi_n^{\pm}+\chi)
=
(I_{8d} \pm \mathfrak{J}) \Psi_n^\pm
=
\begin{bmatrix} 
R_n^\pm \\
\pm J_{2d}\cdot \overline{R_n^\pm} \\
\pm R_n^\pm \\
J_{2d}\cdot \overline{R_n^\pm} 
\end{bmatrix}
\] 
by \eqref{eq306}. In addition, 
\begin{equation} 
\scalebox{0.95}{$
\begin{bmatrix}
\mathfrak{e}_0 & & \pm {}^{\rm t}\overline{\mathfrak{e}_0} & \\
& \mathfrak{e}_0 & & \pm {}^{\rm t}\overline{\mathfrak{e}_0} \\
\pm \mathfrak{e}_{0} & & {}^{\rm t}\overline{\mathfrak{e}_0} & \\
& \pm \mathfrak{e}_{0} & & {}^{\rm t}\overline{\mathfrak{e}_0} 
\end{bmatrix}
 (\Phi_n^{\pm}+\chi) 
 =  
\begin{bmatrix}
\mathfrak{e}_0 & & \pm J_{2d}\cdot {}^{\rm t}\overline{\mathfrak{e}_{0,n}}\cdot J_{2d} & \\
& \mathfrak{e}_{0,n} & & \pm {}^{\rm t}\overline{\mathfrak{e}_0} \\
\pm \mathfrak{e}_{0} & & J_{2d}\cdot {}^{\rm t}\overline{\mathfrak{e}_{0,n}}\cdot J_{2d} & \\
& \pm \mathfrak{e}_{0,n} & & {}^{\rm t}\overline{\mathfrak{e}_0} 
\end{bmatrix}
 (\Phi_n^{\pm}+\chi), 
$}
\end{equation}
where 
$\mathfrak{e}_{0,n}$ on the right-hand side is obtained by replacing $n$ columns 
from the left of $\mathfrak{e}_{0}$ with zero columns, 
since $v_n^{\pm}(n)=v_n^{\pm}(n+1)\cdots=v_n^{\pm}(d-1)=0$ for $1 \leq n \leq d-1$ 
by definition of $\Phi_n^\pm$. 
Therefore, 
\begin{equation} \label{eq309}
\scalebox{0.925}{$
M_\pm
 (\Phi_n^{\pm}+\chi_{8d})
= 
\begin{bmatrix} 
R_n^\pm \\
\pm J_{2d}\cdot \overline{R_n^\pm} \\
\pm R_n^\pm \\
J_{2d}\cdot \overline{R_n^\pm} 
\end{bmatrix} 
~\text{with}~\,\,
M_\pm^\prime = \begin{bmatrix}
\mathfrak{e}_0 & & \mp J_{2d}\cdot {}^{\rm t}\overline{\mathfrak{e}_{0,n}}\cdot J_{2d} & \\
& \mathfrak{e}_{0,n} & & \mp {}^{\rm t}\overline{\mathfrak{e}_0} \\
\mp \mathfrak{e}_{0} & & J_{2d}\cdot {}^{\rm t}\overline{\mathfrak{e}_{0,n}}\cdot J_{2d} & \\
& \mp \mathfrak{e}_{0,n} & & {}^{\rm t}\overline{\mathfrak{e}_0} 
\end{bmatrix}. 
$}
\end{equation}
Here we find that $d$ rows of both $\mathfrak{E}_0 \pm \mathfrak{E}_n^{\sharp}$ 
and $M_\pm$ with indices $(d+1, d+2, \cdots, 2d)$ 
and $n$ rows of both $\mathfrak{E}_0 \pm \mathfrak{E}_n^{\sharp}$ and 
$M_\pm$ with indices $(4d-n+1, 4d-n+2,\cdots, 4d)$ have the same entries. 
Therefore, by comparing $d$ rows of \eqref{eq308} and \eqref{eq309} 
with indices $(d+1, d+2,\cdots, 2d)$, 
we obtain 
$p_n^{\pm}((d-r+1)/2) \pm \overline{q_n^{\pm}((d-r+1)/2)} = 2C_{L}$ 
and 
$p_n^{\pm}(k) \pm \overline{q_n^{\pm}(k)} = 0$ 
for $(d-r+1)/2+1 \leq k \leq (d-r+1)/2+d-1=(3d-r-1)/2$.
Similarly, by comparing $n$ rows of \eqref{eq308} and \eqref{eq309} 
with indices $(4d-n+1, 4d-n+2, \cdots, 4d)$, 
we obtain
$p_n^{\pm}(k) \pm \overline{q_n^{\pm}(k)} = 0$ 
for $-(d+r-1)/2 \leq k \leq -(d+r-1)/2+n-1$. 
\end{proof}

\begin{lemma} \label{lem307} We have 
\begin{equation} 
\aligned 
u_n^{+}(k)  = u_n^{-}(k) \quad (0 \leq k \leq 2d-1), \quad 
v_n^{+}(k)  = -v_n^{-}(k) \quad (-d \leq k \leq d-1)
\endaligned
\end{equation}
for every $0 \leq n \leq d$. 
\end{lemma}
\begin{proof} 
We have $\Phi_n^{\pm} + \chi
=2M_\pm^{-1}\cdot \mathfrak{E}_0\cdot\chi$ with 
\[
M_\pm=
\mathfrak{E}_0 \pm \mathfrak{E}_n^\sharp
= 
\begin{bmatrix}
\mathfrak{e}_0 & & \pm \overline{\mathfrak{e}_{2,n}} & \\
& \mathfrak{e}_0 & & \pm \overline{\mathfrak{e}_{1,n}} \\
\pm J_{2d}\cdot\mathfrak{e}_{1,n}\cdot J_{2d} & & {}^{\rm t}\overline{\mathfrak{e}_0} & \\
& \pm J_{2d}\cdot\mathfrak{e}_{2,n}\cdot J_{2d} & & {}^{\rm t}\overline{\mathfrak{e}_0} 
\end{bmatrix}
\]
by \eqref{eq308}. Put 
\[
A=
\begin{bmatrix} 
\mathfrak{e}_0 & \\
& \mathfrak{e}_0
\end{bmatrix}, \quad 
B=
\begin{bmatrix} 
\pm \overline{\mathfrak{e}_{2,n}} & \\
& \pm \overline{\mathfrak{e}_{2,n}}
\end{bmatrix}, 
\]
\[
C=
\begin{bmatrix} 
\pm J_{2d}\cdot\mathfrak{e}_{1,n}\cdot J_{2d} & \\
& \pm J_{2d}\cdot\mathfrak{e}_{2,n}\cdot J_{2d}
\end{bmatrix}, \quad 
D=
\begin{bmatrix} 
{}^{\rm t}\overline{\mathfrak{e}_0} & \\
& {}^{\rm t}\overline{\mathfrak{e}_0}
\end{bmatrix}. 
\]
Then $\det A=C_{-L}^{2d}\not=0$ (resp. $\det D=\overline{C_{-L}}^{2d}\not=0$) by assumption. 
Therefore the identity for the Schur complement 
$\det M^\pm=\det A\det(D-CA^{-1}B)$ 
(resp. $\det M^\pm=\det D\det(A-BD^{-1}C)$)
shows that $\det(D-CA^{-1}B)\not=0$ (resp. $\det(A-BD^{-1}C)\not=0$), 
since $M^\pm$ are invertible. Also, 
$A-BD^{-1}C$ and $D-CA^{-1}B$ are block-diagonal matrices, 
and thus, their inverse matrices are also block-diagonal. 
Therefore, applying the inversion formula for block matrices  
(\cite[Lemma 3.2]{Su18}) to $M^\pm$, 
we obtain 
\[
\aligned 
\Phi_n^{\pm} + \chi
&= 
\begin{bmatrix}
(A-BD^{-1}C)^{-1}& \mp A^{-1}B(D-CA^{-1}B)^{-1} \\ 
\mp D^{-1}C(A-BD^{-1}C)^{-1} & (D-CA^{-1}B)^{-1}
\end{bmatrix}
\cdot \mathfrak{E}_0\cdot\chi \\
&= 
\begin{bmatrix}
A_{11} & O & \pm A_{13} & O \\ 
O & A_{22} & O & \pm A_{24} \\ 
\pm A_{31} & O & A_{33} & O \\ 
O &\pm A_{42} & O & A_{44} \\ 
\end{bmatrix}
\cdot \mathfrak{E}_0\cdot\chi,
\endaligned 
\]
where $A_{ij}$ are some square matrices of size $2d$. 
Recalling definition \eqref{eq303} of $\Phi_n^\pm$, 
this establishes Lemma \ref{lem307}, since 
all $4d$ entries of the column vector $\mathfrak{E}_0\cdot\chi$ 
with indices $(2d+1,2d+2,\cdots,6d)$ are zero. 
\end{proof}
\begin{lemma} \label{lem308} 
We have 
\begin{equation*} 
p_n^{+}(k) + \overline{q_n^{+}(k)} = p_n^{-}(k) - \overline{q_n^{-}(k)} \quad
 \left(- \frac{d+r-1}{2} \leq k \leq \frac{d-r+1}{2}+n-1\right). 
\end{equation*}
for every $0 \leq n \leq d$, 
where $p_n^{\pm}(k) \pm \overline{q_n^{\pm}(k)}=0$ if 
$-(d+r-1)/2 \leq k \leq -(d+r-1)/2+n-1$ 
or $(d-r+1)/2+1 \leq k \leq (d-r+1)/2+n-1$ 
by Lemma \ref{lem306}.
\end{lemma}
\begin{proof} 
According to Lemma \ref{lem307}, 
we write $u_n(k)=u_n^\pm(k)$ and $v_n(k)= \pm v_n^\pm(k)$. 
By \eqref{eq302}, 
\[
\aligned
{\mathsf E} \, \phi_{n}^{\pm} 
&= {\mathsf E}X(0) \mp {\mathsf E}^{\sharp} Y(0) 
\mp {\mathsf E}^{\sharp} {\mathsf J}{\mathsf P}_n  \phi_{n}^{\pm} \\
&= \sum_{j=0}^{d} C_{L-rj} X((d-r+1)/2-j) \mp 
\sum_{j=0}^{d} \overline{C_{L-rj}} \, Y((d-r+1)/2-j) \\
& \quad -
\sum_{j=0}^{d}\sum_{l=0}^{n-1} \overline{C_{L-rj}} \overline{v_n(l)} X(l-(d+r-1)/2+j) \\
& \quad \mp 
\sum_{j=0}^{d}\sum_{k=0}^{n-1} \overline{C_{L-rj}} \overline{u_n(k)} Y(k+(d-r+1)/2-j), 
\endaligned
\]
where we understand that the double sums on the right-hand side are zero when $n$ is zero.
Therefore, 
\[
\aligned
\mathsf{J}\mathsf{E} \, \phi_{n}^{\pm} 
&=  \mp 
\sum_{j=0}^{d} C_{L-rj} X((d-r+1)/2-j) 
+
\sum_{j=0}^{d} \overline{C_{L-rj}} Y((d-r+1)/2-j)
\\
& \quad \mp 
\sum_{j=0}^{d}\sum_{k=0}^{n-1} C_{L-rj} u_n(k) X(k+(d-r+1)/2-j) \\
& \quad -
\sum_{j=0}^{d}\sum_{l=0}^{n-1} C_{L-rj} v_n(l) Y(l-(d+r-1)/2+j) .
\endaligned
\]
Combining the above, 
\begin{equation} \label{eq310}
\aligned
({\mathsf I} \pm {\mathsf J}){\mathsf E} \, \phi_{n}^{\pm} 
& + ({\mathsf E}X(0) \pm {\mathsf E}^{\sharp} Y(0)) \\
&  = \sum_{j=0}^{d}C_{L-rj} X((d-r+1)/2-j) \pm \sum_{j=0}^{d}\overline{C_{L-rj}} \, Y((d-r+1)/2-j)    \\
& \quad - 
\sum_{j=0}^{d}\sum_{k=0}^{n-1} 
(C_{L-rj} u_n(k) + \overline{C_{-(L-rj)}} \overline{v_n(k)} )
X(k+(d-r+1)/2-j) \\
& \quad \mp 
\sum_{j=0}^{d}\sum_{k=0}^{n-1} 
(\overline{C_{L-rj}} \overline{u_n(k)} + C_{-(L-rj)} v_n(k)) Y(k+(d-r+1)/2-j), 
\endaligned
\end{equation}
where we understand that the double sums on the right-hand side are zero when $n$ is zero.
Comparing the right-hand sides of the above formulas of 
$({\mathsf I} \pm {\mathsf J}){\mathsf E}(\phi_{n}^{\pm}+X(0))$ 
with \eqref{eq307}, 
we obtain Lemma \ref{lem308}. 
\end{proof}

We define the column vectors $A_n^\ast$ and $B_n^\ast$ of length $8d$ by 
\begin{equation} \label{eq311}
\aligned
A_n^\ast=A_n^\ast(\mathcal{C}) &:
= (I+\mathfrak{J}) \Psi_n^+ 
= (I+\mathfrak{J})\mathfrak{E}_0(\Phi_{n}^{+}+\chi), \\
B_n^\ast=B_n^\ast(\mathcal{C}) &:
= (I-\mathfrak{J}) \Psi_n^- 
= (I-\mathfrak{J}) \mathfrak{E}_0(\Phi_{n}^{-}+\chi),
\endaligned
\end{equation}
where $I=I_{8d}$ is the identify matrix of size $8d$. 
We define the row vectors $F^\pm(t,z)$ of length $2d$ by
\[
\aligned 
F^+(t,z)
:= & \begin{bmatrix} 
X(-\tfrac{d+r-1}{2})~X(-\tfrac{d+r-1}{2}+1)~\cdots~X(\tfrac{d-r+1}{2}) 
\quad 0~\cdots~0
\end{bmatrix}, \\
F^-(t,z)
:= & \begin{bmatrix} 
0~\cdots~0 \quad 
Y(\tfrac{d-r+1}{2})~Y(\tfrac{d-r+1}{2}-1)~\cdots~Y(-\tfrac{d+r-1}{2}) 
\end{bmatrix}, 
\endaligned 
\]
and the row vector $F(t,z)$ of length $4d$ by 
\[
F(t,z) := \begin{bmatrix} F^+(t,z) & F^-(t,z) \end{bmatrix}.
\]
Then,  we obtain
\begin{equation} \label{eq312}
\aligned
A_n^\ast(t,z) = \frac{1}{4} \begin{bmatrix} F(t,z) & F(t,z)\end{bmatrix} \cdot A_n^\ast, \quad 
B_n^\ast(t,z) = \frac{i}{4} \begin{bmatrix} F(t,z) & -F(t,z)\end{bmatrix} \cdot B_n^\ast
\endaligned
\end{equation}
by \eqref{eq207}, \eqref {eq306}, and \eqref{eq307}.

\begin{proposition} \label{prop309} 
We have 
\[
- \frac{d}{dt} A_n^\ast(t,z) = z B_n^\ast(t,z), \quad -\frac{d}{dt} B_n^\ast(t,z) = -z A_n^\ast(t,z)
\]
for every $0 \leq n \leq d$. 
\end{proposition}
\begin{proof} 
According to Lemma \ref{lem308}, 
we write 
\[
r_n(k)=p_n^{+}(k) + \overline{q_n^{+}(k)} = p_n^{-}(k) - \overline{q_n^{-}(k)}.
\]
Then, by \eqref{eq307} and definition of $X(k)$ and $Y(l)$, 
\begin{equation*} 
\aligned 
((\mathsf{I} + \mathsf{J})& {\mathsf E} (\phi_{n}^{+}+X(0)))(t,z) \\
& = \sum_{k=-(d+r-1)/2+n}^{(d-r+1)/2}\Bigl( 
r_n(k) e^{i(r(k+1)-1-t)z} + \overline{r_n(k)} e^{-i(r(k+1)-1-t)z} 
\Bigr), 
\endaligned 
\end{equation*}
\begin{equation*} 
\aligned 
((\mathsf{I} - \mathsf{J})& {\mathsf E} (\phi_{n}^{-}+X(0)))(t,z) \\
& = \sum_{k=-(d+r-1)/2+n}^{(d-r+1)/2}\Bigl( 
r_n(k)e^{i(r(k+1)-1-t)z}
-
\overline{r_n(k)} e^{-i(r(k+1)-1-t)z}\Bigr). 
\endaligned 
\end{equation*}
Therefore, the differentiability of $A_n^\ast(t,z)$ and $B_n^\ast(t,z)$ with respect to $t$ is trivial, 
and
\begin{equation*} 
\aligned
-\frac{d}{dt} & 
(({\mathsf I} + {\mathsf J}){\mathsf E}(\phi_{n}^{+}+X(0)))(t,z) \\
&= iz \sum_{k=-(d+r-1)/2+n}^{(d-r+1)/2}\Bigl( 
r_n(k) e^{i(r(k+1)-1-t)z} - \overline{r_n(k)} e^{-i(r(k+1)-1-t)z} 
\Bigr), \\
-\frac{d}{dt} & i 
(({\mathsf I} - {\mathsf J}){\mathsf E}(\phi_{n}^{-}+X(0)))(t,z) \\
&= -z \sum_{k=-(d+r-1)/2+n}^{(d-r+1)/2}\Bigl( 
r_n(k)e^{i(r(k+1)-1-t)z}
+
\overline{r_n(k)} e^{-i(r(k+1)-1-t)z}\Bigr). 
\endaligned
\end{equation*}
Hence we obtain Proposition \ref{prop309} by definition \eqref{eq207}. 
\end{proof}

As mentioned in Section \ref{section_2}, 
the next task is to show the connection formula \eqref{eq208} 
for $A_{n-1}^\ast(t,z)$ and $A_{n}^\ast(t,z)$. 

\begin{proposition} \label{prop310} 
The connection formula \eqref{eq208} holds 
for some real matrix $P_n^\ast$ depending only on $\mathcal{C}$ for all $1 \leq n \leq d$.  
In addition $\det P_n^\ast \not=0$ for all $1 \leq n \leq d$, 
which implies that $H_n$ of \eqref{eq211} is well-defined and $\det H_n=1$. 
\end{proposition}
\begin{proof} 
As in the proof of Lemma \ref{lem308} we write $u_n(k)=u_n^\pm(k)$ and $v_n(k)= \pm v_n^\pm(k)$. 
Taking the limit $t \to rn/2$ in \eqref{eq201}, 
we have 
$X(k):=e^{i(r(k-n/2)+r-1)z}$ and $Y(l)=e^{-i(r(l-n/2)+r-1)z}$. 
Therefore, 
$X(k)=Y(l)$ as a function of $z$ if and only if $n = k+l+r-1$. 

First, we prove \eqref{eq208} for $n \geq 1$. 
Evaluating \eqref{eq310} at $t=rn/2$, 
we get 
\begin{equation} \label{eq313}
\aligned
({\mathsf I} \pm {\mathsf J}){\mathsf E} & 
(\phi_{n}^{\pm} + X(0))(rn/2,z) \\
&  = -(u_n(0)-1)\sum_{j=0}^{d}C_{L-rj} X(0+(d-r+1)/2-j) \\
& \quad - 
\sum_{k=1}^{n-1} (u_n(k) \pm v_n(n-k)) 
\sum_{j=0}^{d} C_{L-rj}  X(k+(d-r+1)/2-j)  \\
& \quad - \overline{v_n(0)} \sum_{j=0}^{d}\overline{C_{L-rj}} \, X(0-(d+r-1)/2+j)  
\pm \left[ \cdots \right] \\
\endaligned 
\end{equation}
and 
\begin{equation*} 
\aligned 
({\mathsf I} \pm {\mathsf J}){\mathsf E} & 
(\phi_{n+1}^{\pm} + X(0))(rn/2,z) \\
&  = - (u_{n+1}(0) \pm v_{n+1}(n)-1 )\sum_{j=0}^{d} C_{L-rj} X(0+(d-r+1)/2-j) \\
& \quad - 
\sum_{k=1}^{n-1} (u_{n+1}^+(k) \pm v_{n+1}^+(n-k)) 
 \sum_{j=0}^{d} C_{L-rj}  X(k+(d-r+1)/2-j)  \\
& \quad \mp
 (\overline{u_{n+1}(n) \pm v_{n+1}(0)}) 
 \sum_{j=0}^{d} \overline{C_{L-rj}} X(0-(d+r-1)/2+j) 
\pm \left[ \cdots \right],
\endaligned
\end{equation*}
where the bracket parts on the right-hand sides 
are the conjugates of the first half of the right-hand sides. 
Therefore, if we prove that the linear relations  
\begin{equation} \label{eq314}
\scalebox{0.9}{$
\aligned 
\alpha_{n+1}^{\ast} &
\begin{bmatrix}
u_n(0)-1 \\
u_n(1) + v_n(n-1) \\
u_n(2) + v_n(n-2) \\
\vdots \\
u_n(n-1) + v_n(1) \\
v_n(0)  
\end{bmatrix}
+ i \beta_{n+1}^{\ast}
\begin{bmatrix}
u_n(0)-1 \\
u_n(1) - v_n(n-1) \\
u_n(2) - v_n(n-2) \\
\vdots \\
u_n(n-1) - v_n(1) \\
-v_n(0)  
\end{bmatrix} 
=
\begin{bmatrix}
u_{n+1}(0) + v_{n+1}(n) -1 \\
u_{n+1}(1) + v_{n+1}(n-1)  \\
u_{n+1}(2) + v_{n+1}(n-2)  \\
\vdots \\
u_{n+1}(n) + v_{n+1}(0) 
\end{bmatrix}, \\
\gamma_{n+1}^{\ast} &
\begin{bmatrix}
u_n(0)-1 \\
u_n(1) + v_n(n-1) \\
u_n(2) + v_n(n-2) \\
\vdots \\
u_n(n-1) + v_n(1) \\
v_n(0)  
\end{bmatrix}
+ i \delta_{n+1}^{\ast}
\begin{bmatrix}
u_n(0)-1 \\
u_n(1) - v_n(n-1) \\
u_n(2) - v_n(n-2) \\
\vdots \\
u_n(n-1) - v_n(1) \\
-v_n(0)  \\
\end{bmatrix} 
= i
\begin{bmatrix}
u_{n+1}(0) - v_{n+1}(n) -1 \\
u_{n+1}(1) - v_{n+1}(n-1)  \\
u_{n+1}(2) - v_{n+1}(n-2)  \\
\vdots \\
u_{n+1}(n) - v_{n+1}(0) 
\end{bmatrix}
\endaligned 
$}
\end{equation}
hold, then they imply \eqref{eq208} for 
$\displaystyle{
P_{n+1}^\ast = 
\begin{bmatrix}
\alpha_{n+1}^\ast & \beta_{n+1}^\ast \\ \gamma_{n+1}^\ast & \delta_{n+1}^\ast 
\end{bmatrix}
}$ and 
\begin{equation} \label{eq315}
\aligned 
\alpha_{n+1}^\ast + i\beta_{n+1}^\ast 
&=\frac{u_{n+1}(0)+v_{n+1}(n)-1}{u_{n}(0)-1} 
= \frac{\overline{u_{n+1}(n)+v_{n+1}(0)}}{\overline{v_n(0)}}, \\
\delta_{n+1}^\ast - i\gamma_{n+1}^\ast 
&= 
\frac{u_{n+1}(0)-v_{n+1}(n)-1}{u_n(0)-1} 
= -\frac{\overline{u_{n+1}(n)-v_{n+1}(0)}}{\overline{v_n(0)}}. 
\endaligned 
\end{equation}
Hence the proof is completed if \eqref{eq314} is shown. 

Subtracting $\mathsf{E}X(0) \pm \mathsf{E}^\sharp Y(0)$ 
from both sides of \eqref{eq302} for $n$ and $n+1$,  
and then taking the limit $t \to rn/2$ on the left-hand sides,  we obtain 
\begin{equation*} 
\aligned 
(({\mathsf E} & \pm {\mathsf E}^{\sharp}{\mathsf J}{\mathsf P_n})(\phi_n-X(0)))(rn/2,z) \\
& = 
(u_{n}(0)-1) \sum_{j=0}^{d} C_{L-rj} \cdot X(0+(d-r+1)/2-j) \\ 
& \quad + 
 \sum_{k=1}^{n-1}  (u_{n}(k) \pm v_{n}(n-k)) 
\sum_{j=0}^{d} C_{L-rj} \cdot X(k+(d-r+1)/r-j) \\
& \quad + \overline{v_n(0)} \sum_{j=0}^{d} \overline{C_{L-rj}} X(0-(d+r-1)/2+j) \\ 
& \quad \pm \overline{(u_{n}(0)-1)} \sum_{j=0}^{d} \overline{C_{L-rj}}\cdot X(n-(d+r-1)/2+j) \\
& \quad \pm 
\sum_{k=1}^{n-1} 
 (\overline{u_{n}(k)} \pm \overline{v_{n}(n-k)}) 
\sum_{j=0}^{d} \overline{C_{L-rj}} \cdot X(n-k-(d+r-1)/2+j) \\
& \quad \pm v_{n}(0) \sum_{j=0}^{d} C_{L-rj} \cdot X(n+(d-r+1)/2-j) \\
& \quad + u_{n}(n) 
\sum_{j=0}^{d} C_{L-rj} \cdot X(n+(d-r+1)/2-j) \\
& \quad + 
 \sum_{k=n+1}^{\infty}  (u_{n}(k) \pm v_{n}(n-k)) 
\sum_{j=0}^{d} C_{L-rj} \cdot X(k+(d-r+1)/2-j) \\
\endaligned 
\end{equation*}
and 
\begin{equation*} 
\aligned 
(({\mathsf E} & \pm {\mathsf E}^{\sharp}{\mathsf J}{\mathsf P_{n+1}}) (\phi_{n+1}-X(0)))(rn/2,z) \\
& = (u_{n+1}(0) \pm v_{n+1}(n)-1) 
\sum_{j=0}^{d} C_{L-rj} \cdot  X(0+(d-r+1)/2-j) \\
& \quad \pm  (\overline{u_{n+1}(0) \pm v_{n+1}(n)-1}) 
\sum_{j=0}^{d} \overline{C_{L-rj}} \cdot X(n-(d+r-1)/2+j) \\
& \quad +
\sum_{k=1}^{\infty} (u_{n+1}(k) \pm v_{n+1}(n-k)) 
\sum_{j=0}^{d} C_{L-rj} \cdot  X(k+(d-r+1)/2-j) \\
& \quad \pm 
\sum_{k=1}^{n} (\overline{u_{n+1}(k) \pm v_{n+1}(n-k)}) 
\sum_{j=0}^{d} \overline{C_{L-rj}} \cdot X(n-k-(d+r-1)/2+j).
\endaligned 
\end{equation*}
In both cases of $n$ and $n+1$, 
the right-hand  sides are 
\[
\aligned 
\mp 2 {\mathsf E}^{\sharp}Y(0)
&= \mp 2\, \sum_{j=0}^{d} \overline{C_{L-rj}} X(n-(d+r-1)/2+j) \\
&= 
\mp 2\, \overline{C_{L}} X(n-(d+r-1)/2) 
\mp 2\, \sum_{j=1}^{d} \overline{C_{L-rj}} X(n-(d+r-1)/2+j). 
\endaligned 
\]
Therefore, by comparing $(n+1)$ coefficient of 
$X(k-(d+r-1)/2)$ for $0 \leq k \leq n$ in equations  
$(({\mathsf E}  \pm {\mathsf E}^{\sharp}{\mathsf J}{\mathsf P_n})(\phi_n-X(0)))(rn/2,z)
= \mp 2 {\mathsf E}^{\sharp}Y(0)$ and 
$(({\mathsf E}  \pm {\mathsf E}^{\sharp}{\mathsf J}{\mathsf P_{n+1}})(\phi_{n+1}-X(0)))(rn/2,z)
= \mp 2 {\mathsf E}^{\sharp}Y(0)$, 
we obtain linear equations 
\begin{equation} \label{eq316}
L_{n+1}^\pm({\mathcal{C}})
\begin{bmatrix}
u_n(0)-1 \\
u_n(1) \pm v_n(n-1) \\
u_n(2) \pm v_n(n-2) \\
\vdots \\
u_n(n-1) \pm v_n(1) \\
\pm v_n(0)  \\
\pm \overline{v_n(0)} \\
\overline{u_n(n-1) \pm v_n(1)} \\
\vdots \\
\overline{u_n(1) \pm v_n(n-1)} \\
\overline{u_n(0)-1} \\
\end{bmatrix}
= \mp 
\begin{bmatrix}
0 \\
\vdots \\
0 \\
2\overline{C_{L}} \\
2C_{L} \\
0 \\
\vdots \\
0 
\end{bmatrix}
-
\begin{bmatrix}
0 \\
\vdots \\
0 \\
C_{-L}u_n(n)\\
\overline{C_{-L}}\,\overline{u_n(n)} \\
0 \\
\vdots \\
0 
\end{bmatrix}
\end{equation}
and 
\begin{equation} \label{eq317}
L_{n+1}^\pm({\mathcal{C}})
\begin{bmatrix}
u_{n+1}(0) \pm v_{n+1}(n) -1 \\
u_{n+1}(1) \pm v_{n+1}(n-1)  \\
u_{n+1}(2) \pm v_{n+1}(n-2)  \\
\vdots \\
u_{n+1}(n) \pm v_{n+1}(0) \\
\overline{u_{n+1}(n) \pm v_{n+1}(0)} \\
\overline{u_{n+1}(n-1) \pm v_{n+1}(1)} \\
\vdots \\
\overline{u_{n+1}(0) \pm v_{n+1}(n) -1}
\end{bmatrix}
= \mp 
\begin{bmatrix}
0 \\
\vdots \\
0 \\
2\overline{C_{L}} \\
2C_{L} \\
0 \\
\vdots \\
0 
\end{bmatrix},
\end{equation}
where $L_n^\pm(\mathcal{C})$ are defined in \eqref{eq101} and \eqref{eq108}, 
and non-zero components of the column vectors 
on the right-hand side are the $(n+1)$th and $(n+2)$th entries. 
Suppose that 
\begin{equation} \label{eq318}
-2\overline{C_{L}}-C_L u_n(n) = K_n \cdot i (2\overline{C_{L}}-C_L u_n(n))
\end{equation}
holds for some $1 \leq n \leq d$ and $K_n \in \R\setminus\{0\}$. 
Then $A_{n}^{\ast}(t,z)=K_{n} B_{n}^{\ast}(t,z)$ by \eqref{eq313} and \eqref{eq316}. 
But, in this case, it must be $K_n= \pm i$ by Proposition \ref{prop309}. 
This is a contradiction. 
Therefore, \eqref{eq318} does not hold for any $K_n \in \R\setminus\{0\}$. 
Hence, there exist real numbers 
$\alpha_{n+1}^\ast$, $\beta_{n+1}^\ast$, $\gamma_{n+1}^\ast$, $\delta_{n+1}^\ast$ 
such that 
\begin{equation} \label{eq319}
\aligned 
\alpha_{n+1}^\ast (-2\overline{C_{L}}-C_L u_n(n))
+ i \beta_{n+1}^\ast (2\overline{C_{L}}-C_L u_n(n)) 
& = -2 \overline{C_{L}}, \\
\gamma_{n+1}^\ast (-2\overline{C_{L}}-C_L u_n(n))
+ i \delta_{n+1}^\ast (2\overline{C_{L}}-C_L u_n(n)) 
& = 2 i \overline{C_{L}} 
\endaligned 
\end{equation}
holds. This implies relation \eqref{eq314}. 

We show that $\displaystyle{
\det P_{n+1}^\ast = 
\begin{bmatrix}
\alpha_{n+1}^\ast & \beta_{n+1}^\ast \\ \gamma_{n+1}^\ast & \delta_{n+1}^\ast 
\end{bmatrix} 
\not=0.
}$ 
If $\det P_{n+1}^\ast=0$, its row vectors are proportional:  
$[\,\alpha_{n+1}^\ast\,\,\,\,\beta_{n+1}^\ast\,]=K_n^\prime [\,\gamma_{n+1}^\ast\,\,\,\,\delta_{n+1}^\ast\,]$, 
say. Then \eqref{eq319} implies $K_n^\prime =-i$, 
but it is impossible for real vectors 
$[\,\alpha_{n+1}^\ast\,\,\,\,\beta_{n+1}^\ast\,]$ and $[\,\gamma_{n+1}^\ast\,\,\,\,\delta_{n+1}^\ast\,]$.  

Finally we prove \eqref{eq208} for $n=0$. 
We have 
$ ({\mathsf I} \pm {\mathsf J}) {\mathsf E}(\phi_{0}^{\pm}+X(0))
= {\mathsf E} X(0) \pm {\mathsf E^\sharp} Y(0)$ 
by $\mathsf{E}\phi_{0}^{\pm}= {\mathsf E} X(0) \mp {\mathsf E^\sharp} Y(0)$, 
since $\mathsf{P}_0=0$.   
Evaluating $({\mathsf I} \pm {\mathsf J}) {\mathsf E}(\phi_{0}^{\pm}+X(0))$ 
and $({\mathsf I} \pm {\mathsf J}) {\mathsf E}(\phi_{1}^{\pm}+X(0))$ 
at $t=0$ by using \eqref{eq310} for $n=0$ and $n=1$, we get
\[
\aligned
({\mathsf I} \pm {\mathsf J}) & {\mathsf E}(\phi_{0}^{\pm}+X(0))(0,z) \\
&  = 
\sum_{j=0}^{d}C_{L-rj} X(0+(d-r+1)/2-j)  \pm 
\sum_{j=0}^{d}\overline{C_{L-rj}} \, X(0-(d+r-1)/2+j), \\
({\mathsf I} \pm {\mathsf J}) & {\mathsf E}(\phi_{1}^{\pm}+X(0))(0,z) \\
&  = 
(1-u_1(0) \mp v_1(0))
\sum_{j=0}^{d}C_{L-rj} X(0+(d-r+1)/2-j) \\ 
& \quad \pm 
(\overline{1- u_1(0) \mp v_1(0)})
\sum_{j=0}^{d}\overline{C_{L-rj}} \, X(0-(d+r-1)/2+j). 
\endaligned
\]
Therefore, \eqref{eq208} holds for 
$\displaystyle{
P_1^\ast = 
\begin{bmatrix}
\alpha_1^\ast & \beta_1^\ast \\ \gamma_1^\ast & \delta_1^\ast 
\end{bmatrix}
}$ 
with 
$\alpha_1^\ast+i\beta_1^\ast= 1 - u_1(0) - v_1(0)$ 
and 
$\gamma_1^\ast+i\delta_1^\ast=1-u_1(0) + v_1(0)$. 
\end{proof}

\begin{lemma} \label{lem311}
Let $\mathfrak{c}$ be the column vector of length $2n$ defined by 
\[
\mathfrak{c}= {}^{\rm t} 
\begin{bmatrix}
0 & \cdots & 0 & \overline{C_{L}} & C_{L} & 0 & \cdots & 0 
\end{bmatrix},
\]
where $\overline{C_L}$ and $C_L$ are $n$th and $(n+1)$th entries, respectively. 
Then, 
\begin{equation} \label{eq320}
\aligned 
\,&
\frac{1}{2}
\Bigl( 
\det L_{n}^+({\mathcal{C}};- 2\mathfrak{c};1)
\det L_{n}^-({\mathcal{C}}; 2\mathfrak{c};2n)  
  +
\det L_{n}^+({\mathcal{C}}; - 2\mathfrak{c}; 2n)
\det L_{n}^-({\mathcal{C}}; 2\mathfrak{c}; 1) \Bigr) \\
& \quad =
\begin{cases}
4|C_{L}|^4 \cdot D_{n-2}({\mathcal{C}})D_n({\mathcal{C}}) & n \geq 2, \\
-4|C_{L}|^2 \cdot D_1(\mathcal{C}) & n=1, 
\end{cases}
\endaligned 
\end{equation}
where 
$L_{n}^\pm({\mathcal{C}}; \mp \mathfrak{c}; k)$ 
is a matrix obtained by replacing 
the $k$th column of $L_{n}^\pm({\mathcal{C}})$ with $\mp \mathfrak{c}$.
Recall that $D_0(\mathcal{C})=1$ by convention. 
\end{lemma}
\begin{proof}
In the case of $n=1$, we have 
\[
\aligned 
\frac{1}{2} & \left(
\det\begin{bmatrix} 
-2\overline{C_{L}} & \overline{C_{L}} \\ 
-2C_{L} & \overline{C_{-L}}
\end{bmatrix}
\det\begin{bmatrix} 
C_{-L} & 2\overline{C_{L}}  \\ 
-C_{L} & 2C_{L}
\end{bmatrix}
+
\det\begin{bmatrix} 
C_{-L} &-2\overline{C_{L}} \\ 
C_{L} & -2C_{L}
\end{bmatrix}
\det\begin{bmatrix} 
2\overline{C_{L}} & -\overline{C_{L}} \\ 
2C_{L} & \overline{C_{-L}}
\end{bmatrix}
\right) \\
&= -4 |C_L|^2 (|C_{-L}|^2-|C_L|^2) = -4 |C_L|^2 D_1(\mathcal{C}). 
\endaligned 
\]

Let $n\geq 2$. 
Multiplying each of the $(n+1)$th to $2n$th columns of 
$\det L_{n}^-({\mathcal{C}}; 2\mathfrak{c};1)$ 
and 
$\det L_{n}^-({\mathcal{C}}; 2\mathfrak{c};2n)$ 
by $-1$, 
and then, multiplying each of the $(n+1)$th to $2n$th rows of them by $-1$, 
\[
\aligned 
\,&
\frac{1}{2}
\Bigl( 
\det L_{n}^+({\mathcal{C}};- 2\mathfrak{c}; 1)
\det L_{n}^-({\mathcal{C}}; 2\mathfrak{c}; 2n)  
  +
\det L_{n}^+({\mathcal{C}};- 2\mathfrak{c}; 2n)
\det L_{n}^-({\mathcal{C}}; 2\mathfrak{c}; 1) \Bigr) \\
& =
2 \, 
\Bigl( 
\det L_{n}^+({\mathcal{C}};\mathfrak{c};1)
\det L_{n}^+({\mathcal{C}};\mathfrak{c}';2n) 
  +
\det L_{n}^+({\mathcal{C}};\mathfrak{c};2n)
\det L_{n}^+({\mathcal{C}};-\mathfrak{c}';1)  \Bigr), \\
\endaligned 
\]
where 
$\displaystyle{
\mathfrak{c}'= {}^{\rm t} 
\begin{bmatrix}
0 & \cdots & 0 & \overline{C_{L}} & -C_{L} & 0 & \cdots & 0 
\end{bmatrix}, 
}$ 
$\overline{C_L}$ and $-C_L$ are $n$th and $(n+1)$th entries, respectively. 
The right-hand side is equal to
\[
4 \vert C_L \vert^2
\Bigl( 
\det L_{n}^+({\mathcal{C}}; e_{n+1}; 1)
\det L_{n}^+({\mathcal{C}}; e_{n}; 2n) 
  -
\det L_{n}^+({\mathcal{C}}; e_{n}; 1)
\det L_{n}^+({\mathcal{C}}; e_{n+1}; 2n)
\Bigr)
\]
by expanding $\det L_{n}^+({\mathcal{C}};\mathfrak{c};1)$ 
and $\det L_{n}^+({\mathcal{C}};-\mathfrak{c}';1)$ along the first columns, 
and 
by expanding $\det L_{n}^+({\mathcal{C}};\mathfrak{c}';2n)$ 
and $\det L_{n}^+({\mathcal{C}};\mathfrak{c};2n)$ along the $2n$th columns, 
Therefore, what should be shown is the equality  
\begin{equation} \label{eq321}
\aligned 
\det & L_{n}^+({\mathcal{C}}; e_{n+1}; 1)
\det L_{n}^+({\mathcal{C}}; e_{n}; 2n)  -
\det L_{n}^+({\mathcal{C}}; e_{n}; 1)
\det L_{n}^+({\mathcal{C}}; e_{n+1}; 2n) \\
& =
|C_{L}|^2 \,D_{n-2}({\mathcal{C}})D_n({\mathcal{C}}). 
\endaligned 
\end{equation}

For a matrix $M$, we denote $M^{(a,b;c,d)}$ 
the matrix removing $a$-th and $b$-th rows 
and $c$-th and $d$-th columns from $M$, and set
\[
\Delta_{n-1}({\mathcal{C}}) 
:= \det\left( L_{n}^+(\mathcal{C})^{(1,n;1,n+1)} \right).
\]
Expanding $\det L_{n}^+({\mathcal{C}}; e_{n+1};1)$ and $\det L_{n}^+({\mathcal{C}}; e_{n};1)$ 
along the $1$st row, 
\[
\det L_{n}^+({\mathcal{C}}; e_{n+1}; 1)
=\overline{C_L}D_{n-1}(\mathcal{C}), 
\quad 
\det L_{n}^+({\mathcal{C}}; e_{n};1)
=\overline{C_L} \Delta_{n-1}(\mathcal{C}),   
\]
because the only non-zero component 
in the $1$st row is $\overline{C_L}$ in the $(n+1)$-th column.
Expanding $\det L_{n}^+({\mathcal{C}}; e_{n}; 2n)$ and $\det L_{n}^+({\mathcal{C}}; e_{n+1};2n)$ 
along the $2n$-th row, 
\[
\det L_{n}^+({\mathcal{C}}; e_{n}; 2n)
= C_L D_{n-1}(\mathcal{C}), 
\quad 
\det L_{n}^+({\mathcal{C}}; e_{n+1}; 2n)
= C_L \overline{\Delta_{n-1}(\mathcal{C})}, 
\]
because the only non-zero component 
in the $2n$-th row is $C_L$ in the $n$-th column.

From the above, the right-hand side of \eqref{eq321} is equal to
\[
|C_L|^2 
\left(
D_{n-1}(\mathcal{C})^2-|\Delta_{n-1}(\mathcal{C})|^2
\right), 
\]
but it is equal to $|C_L|^2 D_{n-2}(\mathcal{C})D_{n}(\mathcal{C})$ 
by \cite[p.41, (12)]{Go86}. Hence we complete the proof.  
\end{proof}

\begin{proposition} \label{prop312}
The matrices $H_n=H_n(\mathcal{C})$ defined by \eqref{eq110} 
are represented by the Schur--Cohn determinants 
as in \eqref{eq114} for all $1 \leq n \leq d$. 
\end{proposition}
\begin{proof} Fix $n$ and write 
$
P_n
= 
(P_1^\ast)^{-1}
\cdots
(P_n^\ast)^{-1}
= [\begin{smallmatrix} a & b \\ c & d \end{smallmatrix}]
$ $(a,b,c,d \in \R)$. 
Then, by \eqref{eq211}, 
\[ 
H_n 
= -\frac{1}{\det P_n} 
\begin{bmatrix}
c^2+d^2 & 
-(ac+bd) \\
-(ac+bd) & 
a^2+b^2
\end{bmatrix} 
= -\frac{1}{\det P_n} \, 
H_n^\prime, 
\]
say. Neither eigenvalue of  $H_n^\prime$ is zero by Proposition \ref{prop310}. 
Furthermore, eigenvalues of $H_n^\prime$ are calculated as 
\[
\frac{1}{2}\left(
a^2+b^2+c^2+d^2 
\pm 
\sqrt{
(a^2+b^2+c^2+d^2)^2 - 4(ad - bc)^2
}
\right),
\]
and  
\[
(a^2+b^2+c^2+d^2)^2 - 4(ad - bc)^2
=((a-d)^2+(b+c)^2)((a+d)^2+(b-c)^2) \geq 0.
\]
Therefore, both eigenvalues of $H_n^\prime$ are positive, 
so $H_n^\prime$ is positive definite. 

On the other hand, by 
$P_n
= 
(P_1^\ast)^{-1}
\cdots
(P_n^\ast)^{-1}
$, 
$P_k^\ast
=
\begin{bmatrix}
\alpha_k^{\ast} & \beta_k^{\ast} \\
\gamma_k^{\ast} & \delta_k^{\ast}
\end{bmatrix}
$, and  
\eqref{eq315}, we obtain 
\[
\frac{1}{\det P_n}
= \prod_{k=1}^{n}(|u_k(0)-1|^2-|v_k(k-1)|^2) 
\left\slash 
\prod_{k=2}^{n}|u_{k-1}(0)-1|^2 
\right., 
\]
because the identity 
\[
|u|^2
\det
\begin{bmatrix}
\Re((z+w)/u) & \Im((z+w)/u) \\
-\Im((z-w)/u) & \Re((z-w)/u)
\end{bmatrix}
=
\det 
\begin{bmatrix}
\Re(z+w) & \Im(z+w) \\
-\Im(z-w) & \Re(z-w)
\end{bmatrix}
= |z|^2-|w|^2
\]
holds for general complex numbers $z,w,u$. 
Hence $H_n$ is written as 
\[
H_n= -\left( \prod_{k=1}^{n}(|u_k(0)-1|^2-|v_k(k-1)|^2) \right) \tilde{H}_n
\]
for some positive definite matrix $\tilde{H}_n$. Thus the proof is completed if 
\begin{equation} \label{eq322}
-\prod_{k=1}^{n}(|u_k(0)-1|^2-|v_k(k-1)|^2) 
= \frac{2^{2n}\,|C_{L}|^{2(2n-1)}}{D_{n-1}(\mathcal{C})D_n(\mathcal{C})} \quad (1 \leq n \leq d)  
\end{equation}
is proved. Applying Cramer's rule to \eqref{eq317},  
\[
\aligned 
u_{k}(0) \pm v_{k}(k-1) -1
& =
\frac{\det L_{k}^\pm({\mathcal{C}}; \mp 2\mathfrak{c};1) }
{D_k(\mathcal{C})}, \\
\overline{u_{k}(0) \pm v_{k}(k-1) -1}
& =
\frac{\det L_{k}^\pm({\mathcal{C}}; \mp 2\mathfrak{c};2k) }
{D_k(\mathcal{C})}. 
\endaligned 
\]
Therefore, 
\[
\aligned 
\,&|u_{k}(0)-1|^2-|v_{k}(k-1)|^2 
 = \frac{1}{2}\Bigl( 
(u_{k}(0) + v_{k}(k-1) -1)\overline{(u_{k}(0) - v_{k}(k-1) -1)}) \\
& \qquad \qquad \qquad \qquad \qquad \qquad \qquad  +
\overline{(u_{k}(0) + v_{k}(k-1) -1)}(u_{k}(0) - v_{k}(k-1) -1)
\Bigr) \\
& = \frac{1}{2D_k(\mathcal{C})^2}
\Bigl( 
\det L_{k}^+({\mathcal{C}}; - 2\mathfrak{c}; 1)
\det L_{k}^-({\mathcal{C}}; 2\mathfrak{c}; 2n) 
 +
\det L_{k}^+({\mathcal{C}}; -2 \mathfrak{c}; 2n)
\det L_{k}^-({\mathcal{C}}; 2\mathfrak{c}; 1) \Bigr). 
\endaligned 
\]
Using \eqref{eq320} on the right-hand side,  
\[
\aligned 
|u_{k}(0)-1|^2-|v_{k}(k-1)|^2
& = 
\begin{cases}
\displaystyle{ 
4\, |C_{L}|^4 \, 
\frac{D_{k-2}(\mathcal{C})}{D_k(\mathcal{C})}}, & k \geq 2, \\[10pt]
\displaystyle{ 
-4\, |C_{L}|^2 \, \frac{1}{D_1(\mathcal{C})} }, & k =1. 
\end{cases}
\endaligned 
\]
This implies \eqref{eq322}. 
\end{proof}

\begin{proposition} \label{prop313}
The pair of functions $(A(t,z),B(t,z))$ of \eqref{eq210} 
satisfies the boundary condition \eqref{eq112}. 
\end{proposition}
\begin{proof} 
The first half of \eqref{eq112} 
follows from definition \eqref{eq210} 
by \eqref{eq208} and \eqref{eq209} for $n=1$, 
since $A_{0}^\ast(0,z)=A(z)$ and $B_0^\ast(0,z)=B(z)$. 
We prove the second half of \eqref{eq112}. 
By definition \eqref{eq207} and Lemma \ref{lem306}, 
\[
A_{d}^\ast(t,z) = C_{L}e^{i(L-t)z} + \overline{C_{L}}e^{-i(L-t)z}, 
\quad 
-i B_{d}^\ast(t,z) = C_{L}e^{i(L-t)z} - \overline{C_{L}}e^{-i(L-t)z}.
\]
Therefore, 
\[
\lim_{t \to L} 
\begin{bmatrix} 
A(t,z) \\ B(t,z)
\end{bmatrix} 
= (P_1^\ast)^{-1} \cdots (P_d^\ast)^{-1} 
\begin{bmatrix} 
 \Re(C_{L}) \\  \Im(C_{L})
\end{bmatrix} 
\]
for fixed $z \in \C$ by definition \eqref{eq210}. 
In particular, the limit is independent of $z$, 
but $A(t,0)$ and $B(t,0)$ are constant function of $t$ 
by Proposition \ref{prop309} and definitions \eqref{eq209} and \eqref{eq210}, 
and hence $A(t,0)=A(0)$ and $B(t,0)=B(0)$. 
\end{proof}

\noindent
{\bf Proof of Theorem \ref{thm_01}.} 
As a summary of the above results, we obtain the following theorem 
which implies Theorem \ref{thm_01}. 

\begin{theorem} \label{thm_04}
Let $\mathcal{C} \in \C^{d+1}$ be as in \eqref{eq105} and define $E=E_\mathcal{C}$ by \eqref{eq106}. 
Suppose that $D_d(\mathcal{C})\not=0$. 
Then, 
\begin{enumerate}
\item $A(t,z)$ and $B(t,z)$ are well-defined and continuous on $[0,L)$  with respect to $t$, 
\item $A(t,z)$ and $B(t,z)$ are continuously differentiable on $(r(n-1)/2,rn/2)$ 
with respect to $t$ for every $1 \leq n \leq d$, 
\item the left-sided limit $\lim_{t \nearrow rn/2}(A(t,z), B(t,z))$ 
defines entire functions of $z$ for every $1 \leq n \leq d$, 
\item $A(t,z)$ and $B(t,z)$ have the forms \eqref{eq113}. 
\item matrices $H_n$ of \eqref{eq110} are well-defined for all $1 \leq n \leq d$ 
and satisfy \eqref{eq114},
\item the pair of functions $(A(t,z),B(t,z))$ defined in \eqref{eq210} satisfies the system \eqref{eq103} 
associated with $H(t)$ defined in \eqref{eq110}, 
\item the pair of functions $(A(t,z),B(t,z))$ satisfies the boundary condition \eqref{eq112}.
\end{enumerate}
\end{theorem}
\begin{proof}
(1), (2), and (3) are consequences of  \eqref{eq209}, \eqref{eq210}, \eqref{eq312}, 
and Proposition \ref{prop310}. 
For (4), we put 
$P_n=
\begin{pmatrix} 
\alpha_n^{\ast\ast} & \beta_n^{\ast\ast} \\ 
\gamma_n^{\ast\ast} & \delta_n^{\ast\ast}
\end{pmatrix}$ and 
\[
a_n(k) = (\alpha_n^{\ast\ast} + i\beta_n^{\ast\ast})r_n(k), \quad 
b_n(k) = (\gamma_n^{\ast\ast} + i\delta_n^{\ast\ast})r_n(k), 
\]
where 
$r_n(k)=p_n^{+}(k) + \overline{q_n^{+}(k)} = p_n^{-}(k) - \overline{q_n^{-}(k)}$ 
as in the proof of Proposition \ref{prop309}. 
Then, we have \eqref{eq113} by \eqref{eq209}, \eqref{eq210}, \eqref{eq312}, 
and the changing of index $k=(L-rj-r+1)/r$. 
(5) follows from Propositions \ref{prop310}, \ref{prop310}, and \ref{prop312}. 
(6) is a consequence of Proposition \ref{prop309}, \eqref{eq209}, and \eqref{eq210}. 
In fact, 
\[
-\frac{d}{dt}
\begin{bmatrix} A_n(t,z) \\ B_n(t,z) \end{bmatrix}
= z \begin{bmatrix} 0 & -1 \\ 1 & 0 \end{bmatrix} 
H_n
\begin{bmatrix} A_n(t,z) \\ B_n(t,z) \end{bmatrix}
\]
for every $r(n-1)/2 \leq t < rn/2$ and $1 \leq n \leq d$ by Proposition \ref{prop309}.  
This implies \eqref{eq103} for $H(t)$ defined by \eqref{eq111}. 
(7) is a consequence of Proposition \ref{prop313}. 
\end{proof}

\section{Proofs of Theorems \ref{thm_02} and \ref{thm_03}
} \label{section_5}

To prove Theorems \ref{thm_02} and \ref{thm_03}, 
we prepare a proposition. 
The proof about it below 
is the almost same as the argument 
in the literature on canonical systems; 
for example, the proof of equation (2.4) and Lemma 2.1, 
and Step 1 of the proof of Theorem 5.1 in Dym \cite{Dym70}. 
However, we purposely give the detailed proof 
to confirm that the positive semidefiniteness of the Hamiltonian, 
which is usually assumed in the theory of canonical systems, 
is not necessary for the proof as well as \cite[Proposition 5.1]{Su18}. 
\begin{proposition} \label{prop401} 
Let $H(t)$ and $(A,B)$ be as in Theorem \ref{thm_02}, 
and write 
$\displaystyle{H_n=\begin{bmatrix} \alpha_n & \beta_n \\ \beta_n & \gamma_n \end{bmatrix}}$ 
for $1 \leq n \leq d$. Then the solution $(A(t,z),B(t,z))$ mentioned in  Theorem \ref{thm_02} 
exists and it is represented as 
\begin{equation} \label{eq401}
\scalebox{0.85}{$
\aligned
\begin{bmatrix}
A(t,z) \\ B(t,z)
\end{bmatrix}
& = 
\begin{bmatrix}
\cos((rn/2-t)z)-\beta_n \sin((rn/2-t)z) & 
- \gamma_n \sin((rn/2-t)z) \\ 
\alpha_n \sin((rn/2-t)z) & 
\cos((rn/2-t)z)+\beta_n \sin((rn/2-t)z)
\end{bmatrix}  \\
& \quad \times
\begin{bmatrix}
\cos((r/2)z)-\beta_{n+1} \sin((r/2)z) & 
- \gamma_{n+1} \sin((r/2)z) \\ 
\alpha_{n+1} \sin((r/2)z) & 
\cos((r/2)z)+\beta_{n+1} \sin((r/2)z)
\end{bmatrix} \\
& \quad \cdots \times
\begin{bmatrix}
\cos((r/2)z)-\beta_d \sin((r/2)z) & 
- \gamma_d \sin((r/2)z) \\ 
\alpha_d \sin((r/2)z) & 
\cos((r/2)z)+\beta_d \sin((r/2)z)
\end{bmatrix} 
\begin{bmatrix}
A \\ B
\end{bmatrix}
\endaligned 
$}
\end{equation}
for $r(n-1)/2 \leq t < rn/2$ and $1 \leq n \leq d$, 
where the product of quadratic matrices on the right-hand side 
consists of only the first matrix if $n=d$.  
In particular, for any $0 \leq t <L$, 
there exists a quadratic matrix-valued function $M(t,z)$ 
consisting of entire functions of $z$ 
such that 
\begin{equation} \label{eq402}
\begin{bmatrix} 
A(t,z) \\ B(t,z)
\end{bmatrix} 
=
M(t,z)
\begin{bmatrix}
A \\ B
\end{bmatrix}, 
\end{equation}
holds and $\det M(t,z)=1$. 
\end{proposition}
\begin{proof} 
By definition, $H(t)$ is integrable on $[t_0, t_1]$ for any $0 \leq t_0 <t_1<L$. 
Hence, 
\begin{equation} \label{eq403}
\scalebox{0.95}{$
\aligned
\begin{bmatrix}
A(t_0,z) \\ B(t_0,z)
\end{bmatrix} 
& =
\left[
I
+
z \int_{t_0}^{t_1}
J(s_1) \, ds_1
+
z^2 \int_{t_0}^{t_1}\int_{s_1}^{t_1}
J(s_1)J(s_2)
\, ds_2 ds_1 \right. \\
&\qquad \qquad \left.
+
z^3  \int_{t_0}^{t_1}\int_{s_1}^{t_1}\int_{s_2}^{t_1}
J(s_1)J(s_2)J(s_3) 
\, ds_3 ds_2 ds_1
+
\cdots 
\right]
\begin{bmatrix}
A(t_1,z) \\ B(t_1,z)
\end{bmatrix},
\endaligned
$}
\end{equation}
where $I=I_2$ and $\displaystyle{
J(t) = 
\begin{bmatrix}
0 & - 1 \\ 1 & 0 
\end{bmatrix}H(t)
}$. 
Taking $C=\max\{ |\alpha_n|,\,|\beta_n|,\,|\gamma_n|~:~1\leq n \leq d\}$ 
and by using the formula 
\[
\int_{t_0}^{t_1}\int_{s_1}^{t_1}\int_{s_2}^{t_1} \cdots \int_{s_{k-1}}^{t_1} 
1 \, ds_k \cdots ds_2 ds_1 = \frac{1}{k!}(t_1-t_0)^k, 
\]
we obtain 
\[
\left|
\left[\int_{t_0}^{t_1}\int_{s_1}^{t_1}\int_{s_2}^{t_1} \cdots \int_{s_{k-1}}^{t_1} 
J(s_1)\cdots J(s_k) \, ds_k \cdots ds_2 ds_1 \right]_{ij} \right| \leq  
2^{k-1}C^k
\frac{1}{k!}(t_1-t_0)^k
\]
for every $1 \leq i,j \leq 2$, where $[M]_{ij}$ means the $(i,j)$-entry of a matrix $M$. 
This estimate implies that 
the right-hand side of \eqref{eq403} 
converges absolutely and uniformly 
if $z$ lies in a bounded region. 
Suppose that 
$H(t)=\Bigl[\begin{smallmatrix} \alpha & \beta \\ \beta & \gamma  \end{smallmatrix}\Bigr]$ 
(a constant matrix) 
with $\alpha\gamma-\beta^2=1$ for $t_0 \leq s \leq t_1$. 
Then the series of  integrals in \eqref{eq403} is calculated as  
\[
\begin{bmatrix}
\cos((t_1-t_0)z)-\beta \sin((t_1-t_0)z) & 
- \gamma \sin((t_1-t_0)z) \\ 
\alpha \sin((t_1-t_0)z) & 
\cos((t_1-t_0)z)+\beta \sin((t_1-t_0)z)
\end{bmatrix}.
\]
Hence we have 
\begin{equation*} 
\scalebox{0.9}{$
\aligned
\, 
&
\begin{bmatrix}
A(t_0,z) \\ B(t_0,z)
\end{bmatrix} 
= 
\begin{bmatrix}
\cos((t_1-t_0)z)-\beta \sin((t_1-t_0)z) & 
- \gamma \sin((t_1-t_0)z) \\ 
\alpha \sin((t_1-t_0)z) & 
\cos((t_1-t_0)z)+\beta \sin((t_1-t_0)z)
\end{bmatrix}
\begin{bmatrix}
A(t_1,z) \\ B(t_1,z)
\end{bmatrix}.
\endaligned
$}
\end{equation*}
Therefore, we obtain \eqref{eq401} for $t \geq r(d-1)/2$ by taking the limit $t_1 \to L$. 
Also, the determinant of the matrix on the right-hand side is
\begin{equation} \label{210319_1}
\aligned 
\det & \begin{bmatrix}
\cos((t_1-t_0)z)-\beta \sin((t_1-t_0)z) & 
- \gamma \sin((t_1-t)z) \\ 
\alpha \sin((t_1-t_0)z) & 
\cos((t_1-t_0)z)+\beta \sin((t_1-t_0)z)
\end{bmatrix} \\
& 
= \cos^2((t_1-t_0)z)+(\alpha\gamma-\beta^2)\sin^2((t_1-t_0)z)=1. 
\endaligned  
\end{equation}
Following the above case, 
applying \eqref{eq403} to $r(d-2)/2 \leq t_0 < r(d-1)/2$ and $t_1=r(d-1)/2$ 
and using the result for $t \geq r(d-1)/2$, 
we obtain \eqref{eq401} for $t \geq r(d-2)/2$. 
By repeating this process, \eqref{eq401} is obtained for all $0 \leq t <L$.  
\end{proof}

\subsection{Proof of Theorem \ref{thm_02}.}  \label{section_5_3} 
To prove \eqref{eq113}, we put 
\begin{equation} \label{0829_4}
\scalebox{0.95}{$\displaystyle{
\begin{bmatrix} 
M_{11}^{n}(z) & M_{12}^{n}(z) \\ 
M_{21}^{n}(z) & M_{22}^{n}(z)
\end{bmatrix} 
=
\prod_{k=n+1}^{d}
\begin{bmatrix}
\cos((r/2)z)-\beta_k \sin((r/2)z) & 
- \gamma_k \sin((r/2)z) \\ 
\alpha_k \sin((r/2)z) & 
\cos((r/2)z)+\beta_k \sin((r/2)z)
\end{bmatrix} 
}$}
\end{equation}
for $1 \leq n \leq d$. 
Then \eqref{eq401} implies 
\begin{equation} \label{eq404}
\scalebox{0.95}{$
\aligned 
A(t,z) & = \cos((rn/2-t)z)\Bigl[ AM_{11}^{n}(z)+BM_{12}^{n}(z) \Bigr] \\
& \quad - \sin((rn/2-t)z) \Bigl[ A(\beta_n M_{11}^{n}(z)+\gamma_n M_{21}^{n}(z))
+ B(\beta_n M_{12}^{n}(z) + \gamma_n M_{22}^{n}(z))\Bigr], \\
B(t,z) & = \cos((rn/2-t)z) \Bigl[ AM_{21}^{n}(z)+BM_{22}^{n}(z) \Bigr] \\
& \quad + \sin((rn/2-t)z) \Bigl[ A(\alpha_n M_{11}^{n}(z)+\beta_n M_{21}^{n}(z)) 
+B(\alpha_n M_{12}^{n}(z) + \beta_n M_{22}^{n}(z))\Bigr] 
\endaligned 
$}
\end{equation}
for $r(n-1)/2 \leq t < rn/2$ and $1 \leq n \leq d$. 
Putting $X=e^{i(r/2)z}$, $X^\ast=e^{i(rn/2-t)z}$, $Y=X^{-1}$, and $Y^\ast=(X^{\ast})^{-1}$, 
we obtain 
\begin{equation} \label{eq405}
M_{rs}^{n}(z) = \sum_{\nu=1}^{d-n} 
\left[ N_{rs}^{n}(\nu)X^{\nu}Y^{d-n-\nu}+ \overline{N_{rs}^{n}(\nu)}X^{d-n-\nu}Y^{\nu} \right] 
\end{equation}
for $r,s \in \{1,2\}$ by induction for $n \geq 1$, 
where $N_{rs}^{n}(\nu)$ are complex numbers depending only on the set $\{H_n\}_{1 \leq n \leq d}$, 
and
\begin{equation} \label{eq406}
\cos((rn/2-t)z) = \frac{1}{2}(X^\ast+Y^\ast), \quad \sin((rn/2-t)z) = -\frac{i}{2}(X^\ast-Y^\ast). 
\end{equation}
Substituting \eqref{eq405} and \eqref{eq406} into \eqref{eq404} 
and then carrying out a simple calculation,  
we obtain \eqref{eq113}. 

By \eqref{eq103}, $A(t,0)$ and $B(t,0)$ are constant function of $t$. 
Hence $E(t,0)=A(t,0)-iB(t,0)=A-iB$ by the boundary condition at $t=L$.  
Suppose that $E(0,z_0)=0$ for some real number $z_0$. 
Then $A(0,z_0)=B(0,z_0)=0$, and thus it should be $(A,B)=(0,0)$ by \eqref{eq402}. 
It is a contradiction. Hence $E(0,z)$ has no real zeros. 
\hfill $\Box$

\subsection{Proof of Theorem \ref{thm_03}.}  \label{section_5_4} 
From \eqref{eq405} the leading term of 
\[
\begin{bmatrix} 
M_{11}^{d-n}(z) & M_{12}^{d-n}(z) \\ 
M_{21}^{d-n}(z) & M_{22}^{d-n}(z)
\end{bmatrix} 
\begin{bmatrix} 
A \\ B
\end{bmatrix} 
\]
with respect to $X$ and $Y$ is written as
\[
\begin{bmatrix} 
P_n X^n + \overline{P_n} Y^n \\ 
Q_n X^n + \overline{Q_n} Y^n
\end{bmatrix}
\]
for some complex numbers $P_n$ and $Q_n$. 
Because 
\[
\scalebox{0.9}{$
\aligned 
\begin{bmatrix} 
M_{11}^{d-n-1}(z) & M_{12}^{d-n-1}(z) \\ 
M_{21}^{d-n-1}(z) & M_{22}^{d-n-1}(z)
\end{bmatrix} 
& =
\begin{bmatrix}
\frac{X+Y}{2}+i\beta_{d-n} \frac{X-Y}{2} & 
 i\gamma_{d-n} \frac{X-Y}{2} \\ 
-i\alpha_{d-n} \frac{X-Y}{2} & 
\frac{X+Y}{2}-i\beta_{d-n} \frac{X-Y}{2}
\end{bmatrix} 
\begin{bmatrix} 
M_{11}^{d-n}(z) & M_{12}^{d-n}(z) \\ 
M_{21}^{d-n}(z) & M_{22}^{d-n}(z)
\end{bmatrix}, 
\endaligned 
$}
\]
we have 
\[
\aligned 
\begin{bmatrix} 
P_{n} \\ Q_{n} 
\end{bmatrix} 
&= \frac{1}{2}
\begin{bmatrix} 
1+i\beta_{d-n+1} & i\gamma_{d-n+1} \\ 
-i\alpha_{d-n+1} & 1-i\beta_{d-n+1}
\end{bmatrix} 
\begin{bmatrix} 
P_{n-1} \\ Q_{n-1} 
\end{bmatrix} \\
&= \frac{1}{2}
\left(
\begin{bmatrix} 
1 & 0 \\ 0 & 1 
\end{bmatrix} 
-i
\begin{bmatrix} 
0 & -1 \\ 1 & 0 
\end{bmatrix} 
\begin{bmatrix} 
\alpha_{d-n+1} & \beta_{d-n+1} \\ 
\beta_{d-n+1} & \gamma_{d-n+1}
\end{bmatrix} 
\right)
\begin{bmatrix} 
P_{n-1} \\ Q_{n-1} 
\end{bmatrix}. 
\endaligned 
\]
The leading term of $E_{d-n+1}(t,z)=A_{d-n+1}(t,z)-iB_{d-n+1}(t,z)$ 
with $t=r(d-n)/2$ is 
\[
(P_n X^n + \overline{P_n} Y^n) -i(Q_n X^n + \overline{Q_n} Y^n). 
\]
Therefore, the coefficient of $X^n$ (resp. $Y^n$) is zero 
if  $(P_n,\,Q_n)$ is proportional to $(1,\,-i)$ (resp. $(1,\,i)$), 
and both are zero if $(P_n,Q_n)=(0,0)$. 
Applying this to $n=d$ gives the desired conclusion. 

The latter half of the theorem is a consequence of 
Schur-Cohn test and Theorems \ref{thm_01} and \ref{thm_02}, 
since $H$ of Theorem \ref{thm_02} must be equal to $H$ of Theorem \ref{thm_01} 
defined for $E_f(z)=e^{irdz/2}f(e^{-irz})$ by Proposition \ref{prop401}. 
\hfill $\Box$

\section{Inductive construction} \label{section_7}

To state the result, we introduce special matrices 
$\mathfrak{P}_n(H)$ and $\mathfrak{Q}_n$ as follows. 
For $n=0$, we define 
\[
\mathfrak{P}_0
=
\mathfrak{P}_0(H)
=
\left[
\begin{array}{cc|cc}
1 & 0 & 0 & 0 \\ 
0 & 1 & 0 & 0 \\ \hline 
0 & 0 & 1 & 0 \\
0 & 0 & 0 & 1 
\end{array}\right], 
\quad 
\mathfrak{Q}_0
=
\left[
\begin{array}{cc|cc|cc|cc}
1 & 0 & 1 & 0 & 0 & 0 & 0 & 0 \\ 
0 & 1 & 0 & 1 & 0 & 0 & 0 & 0 \\ \hline 
0 & 0 & 0 & 0 & 1 & 0 & 1 & 0 \\
0 & 0 & 0 & 0 & 0 & 1 & 0 & 1 
\end{array}\right]. 
\]
For $n \geq 1$ and $\displaystyle{
H=\begin{bmatrix} \alpha  & \beta \\ \beta & \gamma \end{bmatrix}}$, 
we define
\[
\mathfrak{P}_k(H)
=
\left[
\begin{array}{c|c|c|c}
I_{k+2,k+1}^+ & I_{k+2,k+1}^- & \mathbf{0} & \mathbf{0}  \\ \hline
\mathbf{0} & \mathbf{0} & I_{k+2,k+1}^+ & I_{k+2,k+1}^- \\ \hline
(1-i\beta)\cdot\mathbf{0}I_k & \mathbf{0}_{k,k+1} & (-i\gamma)\cdot\mathbf{0}I_k & \mathbf{0}_{k,k+1} \\  
 \mathbf{0}_{k,k+1} & (-i\alpha)\cdot I_k\mathbf{0}  & \mathbf{0}_{k,k+1} & (1-i\beta)\cdot I_k\mathbf{0} \\
\end{array}\right], 
\]
\[
\mathfrak{Q}_k
=
\left[
\begin{array}{c|c|c|c}
I_{k+2} & I_{k+2} & \mathbf{0} & \mathbf{0}  \\ \hline
\mathbf{0} & \mathbf{0} & I_{k+2} & I_{k+2} \\ \hline
\mathbf{0}_{2k,k+2} & \mathbf{0}_{2k,k+2} & \mathbf{0}_{2k,k+2} & \mathbf{0}_{2k,k+2} \\  
\end{array}\right], 
\]
where 
\[
I_{k+2,k+1}^+ = \begin{bmatrix} I_{k+1} \\ \mathbf{0}_{1,k+1} \end{bmatrix}, \quad 
I_{k+2,k+1}^- = \begin{bmatrix} \mathbf{0}_{1,k+1} \\ I_{k+1} \end{bmatrix}, \quad 
\mathbf{0}I_k=\begin{bmatrix} \mathbf{0}_{k,1} & I_k \end{bmatrix}, \quad 
I_k\mathbf{0}=\begin{bmatrix} I_k & \mathbf{0}_{k,1} \end{bmatrix}.
\]
The matrices $\mathfrak{P}_k(H)$ are invertible if $\det H=1$, 
because 
\[
\det \mathfrak{P}_k(H) 
= \det\begin{bmatrix} 1 - i \beta & i \gamma \\ i\alpha & 1 -i\beta \end{bmatrix}^k 
= (\alpha\gamma-(\beta+i)^2)^k. 
\]
Using these matrices, an inductive formula for coefficient of 
$A_n(t,z)$ and $B_n(t,z)$ is described as follows. 

\begin{proposition} \label{prop501} 
Let $\mathcal{C} \in \C^{d+1}$ be as in \eqref{eq105} and define $E=E_\mathcal{C}$ by \eqref{eq106}. 
Suppose that $D_d(\mathcal{C})\not=0$. 
Put 
$r_n(k)=p_n^{+}(k) + \overline{q_n^{+}(k)} = p_n^{-}(k) - \overline{q_n^{-}(k)}$,  
$P_n=
\begin{bmatrix} 
\alpha_n^{\ast\ast} & \beta_n^{\ast\ast} \\ 
\gamma_n^{\ast\ast} & \delta_n^{\ast\ast}
\end{bmatrix}$, and 
\[
a_n(k) = (\alpha_n^{\ast\ast} + i\beta_n^{\ast\ast})r_n(k), \quad 
b_n(k) = (\gamma_n^{\ast\ast} + i\delta_n^{\ast\ast})r_n(k), 
\]
as in Proposition \ref{prop309} and its proof.  
For $0 \leq n \leq d$, 
define the column vectors $A_n^{\ast\ast}$ and $B_n^{\ast\ast}$ of length $d-n+1$  by
\[
\aligned 
A_{n}^{\ast\ast} & = 
~{}^{\rm t}\!\begin{bmatrix} 
a_{n}(\tfrac{d-r+1}{2}) & a_{n}(\tfrac{d-r+1}{2}-1)& \cdots &  a_{n}(-\tfrac{d+r-1}{2} + n)
\end{bmatrix}, \\
B_{n}^{\ast\ast} & = 
~{}^{\rm t}\!\begin{bmatrix} 
b_{n}(\tfrac{d-r+1}{2}) & b_{n}(\tfrac{d-r+1}{2}-1)& \cdots &  b_{n}(-\tfrac{d+r-1}{2} + n)
\end{bmatrix}
\endaligned 
\]
for $1 \leq n \leq d$ and 
\begin{equation} \label{eq501}
A_{0}^{\ast\ast}=B_{0}^{\ast\ast}  = \frac{1}{2} 
~{}^{\rm t}\!\begin{bmatrix} 
C_{(d-r+1)/2} & C_{(d-r+1)/2-1} & \cdots & C_{-(d+r-1)/2} 
\end{bmatrix}. 
\end{equation}
Define the column vectors $\Omega_n$ of length $4(d-n+1)$ by 
\begin{equation} \label{eq502}
\Omega_{n} = 
\begin{bmatrix} A_{n}^{\ast\ast} \\ J_{d-n+1}\overline{A_{n}^{\ast\ast}} \\ 
B_{n}^{\ast\ast} \\ J_{d-n+1}\overline{B_{n}^{\ast\ast}} \end{bmatrix}  \quad (1 \leq n \leq d), 
\quad 
\Omega_0 = 
\begin{bmatrix} A_{0}^{\ast\ast} \\ J_{d+1}\overline{A_{0}^{\ast\ast}} \\ 
B_{0}^{\ast\ast} \\ J_{d+1}\overline{B_{0}^{\ast\ast}} \end{bmatrix}. 
\end{equation}
Then, vectors $\Omega_n$ satisfies the linear relation 
\begin{equation} \label{eq503}
\mathfrak{P}_{d-(n+1)}(H_{n+1}) \Omega_{n+1}=\mathfrak{Q}_{d-(n+1)} \Omega_{n}
\end{equation}
for every $0 \leq n \leq d-1$, where $H_n$ is of \eqref{eq211}.
\end{proposition}
\begin{proof} 
By Lemma \ref{lem308} and \eqref{eq307}, we have
\[
(\mathsf{I} \pm \mathsf{J}){\mathsf E} (\phi_{n}^{\pm}+X(0))
= \sum_{k=-(d+r-1)/2+n}^{(d-r+1)/2}\Bigl( 
r_n(k)X(k) \pm \overline{r_n(k)}Y(k) \Bigr). 
\]
Therefore, 
\[
\aligned 
A_n(t,z)
= \sum_{k=-(d+r-1)/2+n}^{(d-r+1)/2}\Bigl( 
a_n(k)X(k) + \overline{a_n(k)}Y(k) \Bigr), \\
B_n(t,z)
= \sum_{k=-(d+r-1)/2+n}^{(d-r+1)/2}\Bigl( 
b_n(k)X(k) + \overline{b_n(k)}Y(k) \Bigr)
\endaligned 
\]
by \eqref{eq209} and \eqref{eq210}. 
Evaluating these for $n$ and $n+1$ at $t=rn/2$ noting $Y(k)=X(n-k-r+1)$, 
\begin{equation} \label{eq504}
\aligned
A_n(rn/2,z) 
& = 
\sum_{k=-(d+r-1)/2+n}^{(d-r+1)/2} 
\Bigl[
a_n(k) + 
\overline{a_n(n-k-r+1)}
\Bigr]  X(k), \\
B_n(rn/2,z)
& = \sum_{k=-(d+r-1)/2+n}^{(d-r+1)/2} 
\Bigl[
b_n(k) + 
\overline{b_n(n-k-r+1)}
\Bigr]  X(k)
\endaligned
\end{equation}
and 
\begin{equation} \label{eq505}
\aligned
A_{n+1}(rn/2,z)
& = a_{n+1}((d-r+1)/2)\, X((d-r+1)/2) \\
& \quad + \sum_{k=-(d+r-1)/2+n+1}^{(d-r+1)/2-1}
\Bigl[
a_{n+1}(k) 
+ 
\overline{a_{n+1}(n-k-r+1)}) 
\Bigr] \, X(k) \\
& \quad + \overline{a_{n+1}((d-r+1)/2)}\, X(-(d+r-1)/2+n), \\
B_{n+1}(rn/2,z)
& = b_{n+1}((d-r+1)/2)\,X((d-r+1)/2) \\
& \quad + \sum_{k=-(d+r-1)/2+n+1}^{(d-r+1)/2-1}
\Bigl[
b_{n+1}(k) 
+ 
\overline{b_{n+1}(n-k-r+1)}) 
\Bigr] \, X(k) \\
& \quad + \overline{b_{n+1}((d-r+1)/2)}\, X(-(d+r-1)/2+n). 
\endaligned
\end{equation}

On the other hand, by Proposition \ref{prop309},  
\begin{equation} \label{eq506}
\aligned 
\frac{1}{z}\frac{d}{dt}A_n(t,z) 
&= \beta_n A_n(t,z) + \gamma_n B_n(t,z) , \\
-\frac{1}{z}\frac{d}{dt}B_n(t,z) 
&= \alpha_n A_n(t,z) + \beta_n B_n(t,z), 
\endaligned 
\end{equation}
where $H_n
=\begin{bmatrix}
\alpha_n & \beta_n \\ \beta_n & \gamma_n
\end{bmatrix}$. 
Because 
$\displaystyle{
\frac{d}{dt}X(k) = -iz X(k)
}$
and 
$\displaystyle{
\frac{d}{dt}Y(k) = iz Y(k), 
}$ 
the left-hand sides are 
\[
\aligned 
\frac{1}{z}\frac{d}{dt}A_n(t,z)
&= \sum_{k=-(d+r-1)/2+n}^{(d-r+1)/2}\Bigl( 
-ia_n(k)X(k) + i\overline{a_n(k)}Y(k) \Bigr), \\
\endaligned 
\]
\[
\aligned 
\frac{1}{z}\frac{d}{dt}B_n(t,z)
&= \sum_{k=-(d+r-1)/2+n}^{(d-r+1)/2}\Bigl( 
-ib_n(k)X(k) + i\overline{b_n(k)}Y(k) \Bigr). \\
\endaligned 
\]
Therefore, by comparing both sides of \eqref{eq506}, we obtain 
\begin{equation} \label{eq507}
(1- i\beta_n)a_n(k) -  i\gamma_n b_n(k) =0, \quad  
(1+ i\beta_n)b_n(k) + i\alpha_n a_n(k) =0 
\end{equation}
for $-(d+r-1)/2+n \leq k \leq (d-r+1)/2$. 

For $4(d-n)$ complex numbers 
$\{a_{n+1}(k), \overline{a_{n+1}(k)}, b_{n+1}(k), \overline{b_{n+1}(k)} \}_{k=-(d+r-1)/2 + n+1}^{(d-r+1)/2}$, 
we obtain $2(d-n+1)$ linear equations  
by comparing coefficient of $X(k)$ for $-(d+r-1)/2+n \leq k \leq (d-r+1)/2$ 
in equalities $A_{n+1}(rn/2,z)=A_n(rn/2,z)$ and $B_{n+1}(rn/2,z)=B_n(rn/2,z)$ 
by using  \eqref{eq504} and \eqref{eq505}. 
In addition, we obtain $2(d-n-1)$ linear equations 
from differential equations \eqref{eq506} 
by using \eqref{eq507} for $-(d+r-1)/2+n+1 \leq k \leq (d-r+1)/2-1$. 
In total, we obtain $4(d-n)$ linear equations, which is expressed in the form of \eqref{eq503}. 
\end{proof}

The pair of functions $(A(t,z),B(t,z))$ of \eqref{eq210} is written as 
\[
\aligned
A(t,z) &=  \frac{1}{2} \alpha_n^{\ast\ast} \cdot 
\begin{bmatrix} F(t,z) & F(t,z)\end{bmatrix} \cdot (I+\mathfrak{J})\mathfrak{E}_0
(\mathfrak{E}_0 + \mathfrak{E}_n^\sharp)^{-1}\mathfrak{E}_0 \chi \\
& \quad + 
\frac{i}{2} \beta_n^{\ast\ast}  \cdot 
\begin{bmatrix} F(t,z) & -F(t,z)\end{bmatrix} \cdot (I-\mathfrak{J})\mathfrak{E}_0
(\mathfrak{E}_0 - \mathfrak{E}_n^\sharp)^{-1}\mathfrak{E}_0 \chi \\
B(t,z) &= 
\frac{1}{2} \gamma_n^{\ast\ast} \cdot 
\begin{bmatrix} F(t,z) & F(t,z)\end{bmatrix} \cdot (I+\mathfrak{J})\mathfrak{E}_0
(\mathfrak{E}_0 + \mathfrak{E}_n^\sharp)^{-1}\mathfrak{E}_0 \chi \\
& \quad + 
\frac{i}{2} \delta_n^{\ast\ast}  \cdot 
\begin{bmatrix} F(t,z) & -F(t,z)\end{bmatrix} \cdot (I-\mathfrak{J})\mathfrak{E}_0
(\mathfrak{E}_0 - \mathfrak{E}_n^\sharp)^{-1}\mathfrak{E}_0 \chi
\endaligned
\]
for $r(n-1)/2 \leq t < rn/2$ by \eqref{eq209}, \eqref{eq304}, \eqref{eq311}, and \eqref{eq312}. 
These formulas are explicit but it involves the complexity of calculating 
$P_n=
\begin{bmatrix} 
\alpha_n^{\ast\ast} & \beta_n^{\ast\ast} \\ 
\gamma_n^{\ast\ast} & \delta_n^{\ast\ast}
\end{bmatrix}$. 
In contrast, the following method, based on Proposition \ref{prop501}, 
is often useful for computing the triple $(H(t), A(t,z),B(t,z))$. 

\begin{theorem} \label{thm_05} 
Let $\widetilde{\Omega}_{0}$ be a column vector of length $4(d+1)$. 
Define column vectors $\widetilde{\Omega}_{n}$ $(1 \leq n \leq d)$ 
of length $4(d-n+1)$ inductively as follows:
\begin{equation}\label{def_m0}
\aligned 
\widetilde{a}_{n+1}:& =\widetilde{\Omega}_{n}(1) + \widetilde{\Omega}_{n}(d-n+2), \\ 
\widetilde{b}_{n+1}:& =\widetilde{\Omega}_{n}(2(d-n+1)+1) + \widetilde{\Omega}_{n}(3(d-n+1)+1), \\
\widetilde{\alpha}_{n+1} &:= \frac{|\widetilde{b}_{n+1}|^2}{\Re(\,\widetilde{a}_{n+1}\overline{(i\widetilde{b}_{n+1})}\,)}, \quad 
\widetilde{\beta}_{n+1} := \frac{\Im(\,\widetilde{a}_{n+1}\overline{(i\widetilde{b}_{n+1})}\,)}
{\Re(\widetilde{a}_{n+1}\overline{(i\widetilde{b}_{n+1})}\,)}, \\ 
\widetilde{\gamma}_{n+1} &:= \frac{|\widetilde{a}_{n+1}|^2}
{\Re(\,\widetilde{a}_{n+1}\overline{(i\widetilde{b}_{n+1})}\,)}, 
\endaligned 
\end{equation}
\begin{equation}\label{eq508}
\widetilde{H}_{n+1}
:= \begin{bmatrix} 
\widetilde{\alpha}_{n+1} & \widetilde{\beta}_{n+1} \\
\widetilde{\beta}_{n+1} & \widetilde{\gamma}_{n+1}
\end{bmatrix}, 
\end{equation}
\begin{equation}\label{eq509}
\widetilde{\Omega}_{n+1} := (\mathfrak{P}_{d-(n+1)}(\widetilde{H}_{n+1}))^{-1}
\mathfrak{Q}_{d-(n+1)} \, \widetilde{\Omega}_{n},
\end{equation}
where $\mathfrak{P}_0(\tilde{H}_0):=\mathfrak{P}_0$ 
and $v(j)$ means the $j$-th component of a column vector $v$. 

Suppose that $\widetilde{\Omega}_{0}$ is the vector 
defined by \eqref{eq501} and \eqref{eq502} 
for a vector $\mathcal{C} \in \C^{d+1}$ as in \eqref{eq105} such that 
$D_d(\mathcal{C})\not=0$. 
Then $\widetilde{H}_n$ and $\widetilde{\Omega}_{n}$ are well-defined as functions of $\mathcal{C}$ 
for every $1 \leq n \leq d$, and  
\[
H_n=\widetilde{H}_n ,\quad \Omega_n=\widetilde{\Omega}_n,
\]
where $H_n$ and $\Omega_n$ are defined in \eqref{eq211} and \eqref{eq502}, respectively. 
\end{theorem}
\begin{proof}
Solving \eqref{eq507} for fixed $k$, 
\[
\aligned 
\alpha_n &= \frac{|b_n(k)|^2}{\Re(\,a_n(k)\overline{(ib_n(k))}\,)}, \quad 
\beta_n = \frac{\Im(\,a_n(k)\overline{(ib_n(k))}\,)}{\Re(\,a_n(k)\overline{(ib_n(k))}\,)}, \quad 
\gamma_n = \frac{|a_n(k)|^2}{\Re(\,a_n(k)\overline{(ib_n(k))}\,)}. 
\endaligned 
\]
Therefore, $H_n$ and $\Omega_{n}$ of \eqref{eq211} and \eqref{eq502} 
satisfy \eqref{eq508} and \eqref{eq509} by the definitions of 
$\mathfrak{P}_k(H_k)$, $\mathfrak{Q}_k$, and \eqref{eq503}. 
Therefore, $H_n \not \equiv 0$ as a function of $\mathcal{C}$ for every $1 \leq n \leq d$ 
by Theorem \ref{thm_01}, 
since all roots of the derivative of the cyclotomic polynomial of degree $d+1$ 
lie inside the unit circle. 
Hence, the invertibility of  $\mathfrak{P}_k(H_k)$ implies that 
$\widetilde{\Omega}_{1},\widetilde{\Omega}_{2}, \cdots, \widetilde{\Omega}_{d}$ 
and 
$\widetilde{H}_{1},\widetilde{H}_{2}, \cdots, \widetilde{H}_{d}$ 
are uniquely determined from the initial vector $\widetilde{\Omega}_{0}$. 
Therefore, $\Omega_{n}=\widetilde{\Omega}_{n}$ and $H_{n}=\widetilde{H}_{n}$ 
for every $1 \leq n \leq d$ if  $\tilde{\Omega}_{0}=\Omega_{0}$. 
\end{proof}

By definition of the matrices $\mathfrak{P}_k(H_k)$,  
in \eqref{eq503}, $\Omega_{n+1}(1)$, $\Omega_{n+1}(2(d-n))$, 
$\Omega_{n+1}(2(d-n)+1)$,  and $\Omega_{n+1}(4(d-n))$
are determined from $\Omega_n$ independent of $H_{n+1}$.
Hence, we can define $\Omega_{n}$ by 
taking 
\[
\Omega_{n}^\prime=\mathfrak{P}_{d-n}(H_{n})^{-1}\mathfrak{Q}_{d-n} \, \Omega_{n-1}
\] 
for $\Omega_{n-1}$ 
and then substituting $H_n=\begin{bmatrix} 
\alpha_{n} & \beta_{n} \\
\beta_{n} & \gamma_{n}
\end{bmatrix}$ 
defined by 
\[
\aligned 
\alpha_n &= \frac{|\Omega_{n+1}^\prime(2(d-n)+1)|^2}
{\Re(\,\Omega_{n+1}^\prime(1)\overline{(i\Omega_{n+1}^\prime(2(d-n)+1))}\,)}, \quad 
\beta_n = \frac{\Im(\,\Omega_{n+1}^\prime(1)\overline{(i\Omega_{n+1}^\prime(2(d-n)+1))}\,)}
{\Re(\,\Omega_{n+1}^\prime(1)\overline{(i\Omega_{n+1}^\prime(2(d-n)+1))}\,)}, \\
\gamma_n &= \frac{|\Omega_{n+1}^\prime(1)|^2}{\Re(\,\Omega_{n+1}^\prime(1)\overline{(i\Omega_{n+1}^\prime(2(d-n)+1))}\,)}. 
\endaligned 
\]
 into $H_n$ of $\Omega_{n}^{\prime}$. 
In this way we can inductively obtain vectors $\Omega_1,\dots, \Omega_n$ 
and quadratic real symmetric matrices $H_1,\dots, H_d$ starting with 
the initial vector $\Omega_0$.

%


\bigskip \noindent
\\
Department of Mathematics, 
School of Science, \\
Tokyo Institute of Technology \\
2-12-1 Ookayama, Meguro-ku, 
Tokyo 152-8551, JAPAN  \\
Email: {\tt msuzuki@math.titech.ac.jp}

\end{document}